\documentclass[10pt]{report}
\usepackage[fleqn]{amsmath}
\usepackage{amsthm,amsmath,amsfonts,amssymb}
\usepackage[all]{xy}
\newcommand{\RR}{{\mathbb R}}

\newcommand{\e}{{\mathbb S}}
\newcommand{\ZZ}{{\mathbb Z}}
\newcommand{\VV}{{\mathbb V}}
\newcommand{\EE}{{\mathbb E}}
\newcommand{\CC}{{\mathbb C}}
\newcommand{\GG}{{\mathbb G}}
\newcommand{\TT}{{\mathbb T}}
\newcommand{\dd}{\mathrm d}
\newcommand\cF{\mathcal F}

\newcounter{mylc}
\renewcommand{\themylc}{\roman{mylc}}

\newtheorem{theorem}{Theorem}[section]
\newtheorem{proposition}[theorem]{Proposition}
\newtheorem{corollary}[theorem]{Corollary}
\newtheorem{lemma}[theorem]{Lemma}
\newtheorem{example}[theorem]{Example}
\theoremstyle{remark}
\newtheorem{remark}[theorem]{Remark}
\theoremstyle{conjecture}
\theoremstyle{definition}

\newtheorem{question}[theorem]{Question}
\newtheorem{definition}[theorem]{Definition}
\newtheorem{notation}[theorem]{Notation}
\newtheorem{notations}[theorem]{Notations and conventions}
\newtheorem{const}[theorem]{Construction}
\newtheorem{appl}[theorem]{Application}
\newtheorem{problem}[theorem]{Problem}
\newtheorem{claim}[theorem]{Claim}
\theoremstyle{construction}

\usepackage{color}
\definecolor{darkgreen}{cmyk}{1,0,1,.2}
\definecolor{m}{rgb}{1,0.1,1}
\definecolor{green}{cmyk}{1,0,1,0}

\definecolor{test}{rgb}{1,0,0}
\definecolor{cmyk}{cmyk}{0,1,1,0}

\begin{document}
\title{\bf Integrable embeddings and foliations}
\author{Gilbert Hector\thanks{gilberthector@orange.fr}\\ Institut C. Jordan, UMR CNRS 5028\\
Universit\'e C. Bernard (Lyon 1)\\
69622 Villeurbanne-Cedex (France)
 \and Daniel Peralta-Salas\thanks{dperalta@icmat.es}\\ Instituto de Ciencias Matem\'aticas\\
CSIC-UAM-UC3M-UCM\\
C/ Serrano 123, 28006 Madrid (Spain)}

\maketitle
\begin{abstract}
A $k$-submanifold $L$ of an open $n$-manifold $M$ is called \textit{weakly integrable (WI)} [resp. \textit{strongly integrable (SI)}] if there exists a submersion $\Phi:M\to \EE^{n-k}$ such that $L\subset \Phi^{-1}(0)$ [resp. $L= \Phi^{-1}(0)$]. In this work we study the following problem, first stated in a particular case by Costa et al. (Invent. Math. 1988): which submanifolds $L$ of an open manifold $M$ are WI or SI?

For general $M$, we explicitly solve the case $k=n-1$ and provide necessary and sufficient conditions for submanifolds to be WI and SI in higher codimension. As particular cases we recover the theorem of Bouma and Hector (Indagationes Math. 1983) asserting that any open orientable surface is SI in $\EE^3$, and Watanabe's and Miyoshi's theorems (Topology 1993 and 1995) claiming that any link is WI in an open $3$-manifold. In the case $M=\EE^n$ we fully characterize WI and SI submanifolds, we provide examples of $3$- and $7$-manifolds which are not WI and we show that a theorem by Miyoshi (Topology 1995) which states that any link in $\EE^3$ is SI does not hold in general. The right analogue to Miyoshi's theorem is also proved, implying in particular the surprising result that no knot in $\EE^3$ is SI.

Our results applied to the theory of foliations of Euclidean spaces give rise to some striking corollaries: using some topological invariants we classify  all the submanifolds of $\EE^n$ which can be realized as proper leaves of foliations; we prove that $\e^3$ can be realized as a leaf of a foliation of $\EE^n$, $n \geq 7$, but not in $\EE^5$ or $\EE^6$, which partially answers a question by Vogt (Math. Ann. 1993); we construct open $3$-manifolds which cannot be leaves of a foliation of any compact $4$-manifold but are proper leaves in $\EE^4$.

The theory of WI and SI submanifolds is a framework where many classical tools of differential and algebraic topology play a prominent role: Phillips-Gromov h-principle, Hirsh-Smale theory, complete intersections, Seifert manifolds, the theory of immersions and embeddings, obstruction theory ...
\end{abstract}

\tableofcontents
\chapter{Introduction  and partial description of results}\label{intro}

\section{Introduction}
The theory of immersions and embeddings of manifolds is a classical subject in differential topology. The standard setting consists of an ambient $n$-manifold $M$ and a $k$-manifold $L$, $k<n$, to be embedded. Unless otherwise stated we will always assume that the following more precise conditions are fulfilled.

\begin{notations}

(1) All considered manifolds and maps are smooth and oriented or orientation preserving. A manifold will be called \textit{open} if all its connected components are non-compact but possibly with boundary. In particular our ambient manifold $M$ will be open, orientable, connected and mostly without boundary.

(2) The manifolds $L$ to be embedded will be without boundary but not necessarily compact nor connected. The \textit{embeddings} $h:L \to M$ that we shall consider are one-to-one proper immersions of $L$ into $M$. For the sake of notational simplicity, we will often identify the manifold $L$ with its embedded image $h(L)\subset M$ and the term \textit{submanifold} will usually mean embedded submanifold as described above.

(3) The Euclidean $n$-space is denoted by $\EE^n$ and the $n$-vector space by $\RR^n$.
\end{notations}

It is well known that an embedded submanifold $h(L)\subset M$ can be described as the zero set of a function $f\in C^\infty(M,\EE)$. Moreover, if the normal bundle of the embedding is trivial, the submanifold is contained in the zero set of a map $\Phi\in C^\infty(M,\EE^{n-k})$ such that $\dd \Phi|_{h(L)}$ has maximal rank. In general, $\Phi$ is not a submersion from $M$ to $\EE^{n-k}$ so a natural problem is to obtain necessary and/or sufficient conditions over $L$ and $h$ under which the singularities of $\Phi$ can be eliminated thus producing a submersion.

The aim of this paper is to study this problem, first stated in \cite{CGG88} for the particular case where $M = \EE^n$. The general question is which submanifolds of an open manifold $M$ can be realized as fibers of a submersion $\Phi: M \to \EE^{n-k}$. To deal with this question in a general and systematic way we introduce the following definition.

\begin{definition}\label{D:WI}
An embedding $h:L \to M$ is called {\it weakly integrable (WI)} if there exists a smooth submersion $\Phi:M\to \EE^{n-k}$ such that $h(L)\subset \Phi^{-1}(0)$ and {\it strongly integrable (SI)} if $h(L)=\Phi^{-1}(0)$. The submersion $\Phi$ is a {\it weak or strong equation} of $h(L)$; it is of course not unique but provides a natural trivialization of the normal bundle $\nu(h)$ of $h(L)$.
\end{definition}

As said before, we will often identify $L$ with $h(L)$. In particular, we shall say that $L$ is WI or SI when we really refer to its embedding in $M$.

In \cite{CGG88} the authors constructed an example of a submersion $\Phi:\EE^3\to \EE^2$ with a fiber having a connected component diffeomorphic to $\e^1$, but the question whether any link in $\EE^3$ is WI or SI was unanswered. It was partially
solved by Watanabe \cite{Wa93}, who proved that any
link in $\EE^3$ is WI, and completely solved by Miyoshi
\cite{Mi95}, who showed that in fact any link in $\EE^3$ is SI. Miyoshi also analyzed compact $1$-dimensional WI and SI submanifolds of general open $3$-manifolds. Regrettably, we shall show that Miyoshi's proof contains a substantial gap and in fact his main theorem is wrong (we will give counterexamples to it and the correct statement of the theorem).

The literature concerning weak or strong integrability of higher dimensional submanifolds is rather scarce. For example, the fact that any open surface $L$ embedded into $\EE^3$ is SI can be extracted from \cite{HB83}, although the authors did not state it in this way. Our goal in this work is to fill this gap and provide a general theory of WI and SI embeddings. In particular, we shall recover and extend all the previous results \cite{HB83,Wa93,Mi95} as straightforward corollaries of more general theorems. 
The importance of the study of WI and SI embeddings lies also in the fact that it provides a good framework for the interplay between classical subjects of differential and algebraic topology: the Phillips-Gromov h-principle \cite{Ph67,Gr69}, complete intersections \cite{BK96,Fo01}, Seifert submanifolds \cite{Ro75,KW77} and the theory of immersions into Euclidean spaces developed by Thom \cite{Th52}, Milnor \cite{Mi56}, Kervaire \cite{Ke59,Ke59_2} and many other authors. In particular, the Hirsch-Smale immersion theorem \cite{Sm59,Hirsch59} will be one of the basic tools we will use. We shall provide a new interpretation of all these topics in our context, and we will eventually obtain some previously well known results as corollaries.

\section{Most relevant results}\label{S.mostrelevant}

Our discussion in the paper will deal with embeddings of any $k$-manifold $L$ in any codimension. But as it would be too long and tedious to describe all corresponding results in this introduction, we restrict here to embeddings of compact connected manifolds in codimension $m\geq 2$. Indeed these are in some sense the most relevant situations to be studied. Also we describe WI embeddings more than SI, the latter being related, as we shall see later, to the ``complete intersection'' embeddings in the sense of Bochnak and Kucharz (see \cite{BK96}).

With these restrictions in mind, our main results can be summarized as follows:

\medskip

\noindent \textbf{Theorem A}. \textit{Let $L \subset M$ be a compact and connected submanifold of codimension $m \geq 2$ in the open $n$-manifold $M$. The following statements are equivalent:
\begin{enumerate}
\item $L$ is WI.
\item The normal bundle $\nu(L)$ of $L$ extends over $M$ as a trivial bundle.
\end{enumerate}}

\noindent \textbf{Corollary B}. \textit{If $M$ is contractible, e.g. $M = \EE^n$, we get three equivalent conditions:
\begin{enumerate}
\item $L$ is WI.
\item The normal bundle $\nu(L)$ of $L$ extends over $M$.
\item The tangent bundle $T(L)$ of $L$ extends over $M$ and thus $L$ is parallelizable.
\end{enumerate}}

The discussion in the case $M = \EE^n$ is much more subtle. Here arises the question whether the obvious preliminary conditions that $\nu(L)$ and $T(L)$ are trivial are sufficient to ensure that $L$ is WI. For example, we will show that there exist manifolds for which the fact that an embedding $h$ of $L$ is WI or not is an intrinsic character of the manifold not depending on the embedding $h$.

\begin{definition} A parallelizable manifold $L$ is called \textit{totally weakly integrable (TWI)} if for any $n \geq k+2$, any embedding $h$ of $L$ into $\EE^n$ is WI provided that its normal bundle is trivial.
\end{definition}

We obtain the following rather surprising result:
\medskip

\noindent \textbf{Theorem C}. \textit{For $k \notin \{3,7\}$, any compact connected parallelizable $k$-manifold  is TWI. For $k\in\{3,7\}$, a compact connected parallelizable $k$-manifold is TWI if and only if $\chi^*(L)=0$ mod. $2$, where $\chi^*(L)$ is the semicharacteristic of $L$.}

\medskip
This theorem allows us to classify the compact parallelizable manifolds which are not TWI, for instance, the spheres $\e^3$ and $\e^7$.

\medskip
\noindent \textbf{Theorem D}.\textit{ For $k\in \{3,7\}$, the sphere $\e^k$ is not TWI, moreover we have the following:
\begin{enumerate}
\item For $n \geq 2k+1$, any embedding of $\e^k$ in $\EE^n$ has trivial normal bundle and is WI.
\item For $k+2\leq n \leq 2k$, no embedding of $\e^k$ in $\EE^n$ is WI, whether its normal bundle is trivial or not.
\end{enumerate}}

\begin{remark}
It is known that any embedding of $\e^k$ in $\EE^n$, $n\geq k+1$ and $k\in\{3,7\}$, has trivial normal bundle, except for the embeddings of $\e^7$ into $\EE^{11}$.
\end{remark}

Concerning SI embeddings, we only mention the following result, which in particular provides counterexamples (any knot in $\EE^3$) to a theorem by Miyoshi~\cite[Theorem 1]{Mi95}.

\medskip
\noindent \textbf{Theorem E}. \textit{No embedding of $\e^k$, $k\in\{1,3,7\}$, in $\EE^{2k+1}$ is SI. In particular, no knot in $\EE^3$ is SI.}

\subsubsection*{Applications to foliation theory}

In particular, we shall see that a submanifold $L\subset \EE^n$ can be realized as a proper leaf of a foliation of $\EE^n$ if and only if it is WI. Restricting again to compact manifolds, we get:

\medskip
\noindent \textbf{Theorem F}. \textit{Let $L \subset \EE^n$ be a compact connected parallelizable $k$-submanifold of codimension $m \geq 2$. The following properties hold true:
\begin{enumerate}
\item If $n\leq 2k$, $L$ is a leaf of a foliation of $\EE^n$ if and only if $k\notin\{3,7\}$ and it has trivial normal bundle or $k\in\{3,7\}$, it has trivial normal bundle and $\chi^*(L)=0$ mod. $2$ (here $\chi^*(L)$ denotes de semicharacteristic of $L$).
\item If $n\geq 2k+1$, $L$ is a leaf of a foliation of $\EE^n$ if and only if it has trivial normal bundle.
\end{enumerate}}
\noindent\textbf{Corollary G}. \textit{For $k \in \{3,7\}$, the sphere $\e^k$ can be realized as a leaf in $\EE^n$ if and only if $n\geq 2k+1$.}
\medskip

This result provides a partial answer to a question of Vogt \cite{V93}. Moreover, we obtain a complete characterization of submanifolds which can be realized as proper leaves in $\EE^n$. For example, we shall construct an open $3$-manifold which cannot be a leaf of any codimension one foliation of a closed $4$-manifold~\cite{Gh85,In85}, but can be embedded SI in $\EE^{4}$ and therefore realized as a leaf of a codimension $1$ foliation of $\EE^4$.

\section{Organization of the paper}

To finish let us briefly describe the organization of the paper. It consists of two main parts. In the first one (Chapter \ref{Ch:gen}) we consider the general aspects of the theory and study WI and SI embeddings in any open manifold. In the second part (Chapter \ref{Ch:Euc}) we focus on embeddings into Euclidean spaces and provide several applications of our results to the theory of foliations. A final section (Chapter \ref{Ch:final}) concerning real analytic submanifolds and open problems is also included.

More specifically, in Section \ref{S:pre}, we discuss some simple general properties of WI and SI embeddings as well as the relationship with Seifert manifolds and complete intersections. In Section \ref{1} we classify WI and SI codimension one submanifolds of open manifolds. Codimension $m\geq 2$ submanifolds are considered in Section \ref{s.general} where the Phillips-Gromov h-principle is introduced to study WI embeddings. In Section \ref{s.cod2} complete intersections are used to characterize SI embeddings and the codimension 2 case is studied in detail. Embeddings in Euclidean spaces are introduced in Section \ref{Euc0}, and some general properties as well as the Hirsch-Smale theorem are discussed in Section \ref{Euc1}. The tangential characteristic class is defined in Section \ref{S:tangent}, where WI submanifolds in $\EE^n$, $n\geq 2k+1$, are characterized. In Sections \ref{S:normal}, \ref{S:TWI} and \ref{S:main} we focus on the normal characteristic class and, using the machinery of the immersion theory, we discuss WI submanifolds of dimension $\neq 3,7$. The remaining cases, i.e. $3$- and $7$-manifolds, are studied in Section \ref{Euc4}. SI embeddings in $\EE^n$ are analyzed in Sections~\ref{Euc5} and~\ref{S:miyoshi} using the previous results and the theory of complete intersections. In particular, we shall give counterexamples to Miyoshi's theorem and we shall characterize the SI links of $\EE^3$, thus proving the right analogue of Miyoshi's theorem. Finally, in Section \ref{Euc6}, we apply our results to the theory of foliations of Euclidean space and to the theory of critical sets.

\section*{Acknowledgements}

The authors are very grateful to V. Borrelli, S. Miyoshi,
T. Mostowski, F. Presas and A. Weber for their interesting remarks on
preliminary versions of this work. We are particularly indebted to M. Takase for
his enlightening comments regarding Hirsch-Smale theory and for explaining to us how to classify the $3$- and $7$-manifolds which are totally weakly integrable using the semicharacteristic. Special thanks are due to F. Gonz\'alez-Gasc\'on for introducing the second author to this problem several years ago; this paper is dedicated to him on the occasion of his $65$th birthday. D.P.-S. acknowledges the financial support of the Spanish MICINN through the Ram\'on y Cajal program and
the grant no.~MTM2007-62478.

\chapter{General theory}\label{Ch:gen}
\section{Preliminary properties}\label{S:pre}

A submersion $\Phi:M\to \EE^m$ defines a foliation $\cF_{\Phi}$ of $M$ whose leaves are the connected components of the level sets $\Phi^{-1}(c)$, $c\in \Phi(M)$. This foliation is without holonomy, its leaf-space is a possibly non-Hausdorff manifold of dimension $m$ and we say that $\cF_{\Phi}$ is \textit{simple}. In particular, any WI submanifold $L \subset M$ is, by definition, a union of leaves of the corresponding simple foliation $\cF_{\Phi}$ for any equation $\Phi$ of $L$ and the following elementary properties hold:

\begin{lemma}\label{L:easy} If $L\subset M$ is a WI submanifold, then
\begin{enumerate}
\item The normal bundle $\nu(L)$ of $L$ extends over $M$ as a trivial bundle $\tilde \nu(L)$.
\item The tangent bundle $T(L)$ of $L$ extends over $M$, and if $M$ is contractible, its extension $\tilde T(L)$ is trivial and so $L$ is parallelizable.
\end{enumerate}
\end{lemma}
\begin{proof} Let $\Phi$ be a weak equation of $L$, then $\nu(L)$ and $T(L)$ are extended respectively by the normal bundle $\nu(\cF_\Phi)$ and the tangent bundle $T(\cF_\Phi)$ of the foliation $\cF_\Phi$. Moreover, $\nu(\cF_\Phi)$ is trivialized by the bundle map $\dd\Phi: \nu(\cF_\Phi) \to T(\EE^{n-k})$, and if $M$ is contractible, $T(\cF_\Phi$), which is defined over all of $M$, is trivial and so is its restriction $T(L)$ to $L$. The proof is completed by setting $\tilde \nu(L):=\nu(\cF_\Phi)$ and $\tilde T(L):=T(\cF_\Phi)$.
\end{proof}

The triviality of the extended bundle $\tilde \nu(L)$ shows that a straightforward obstruction for the existence of WI embeddings of a $k$-manifold $L$ in an $n$-manifold $M$ is that $M$ must admit a global tangent $(n-k)$-frame; in the usual terminology $\text{span}(M)\geq n-k$.

Next we recall the definitions of two well known concepts which are related to WI and SI submanifolds.

\begin{definition}\label{complete}
Let $M$ be an open manifold of dimension $n$. A $k$-dimensional submanifold $L$ of $M$ is a {\em complete intersection (CI)} if there exists a smooth map $\Xi:M \to \EE^{n-k}$ such that $0$ is a regular value of $\Xi$ and $L=\Xi^{-1}(0)$, cf. e.g.~\cite{Fo01}. The map $\Xi$, which is of course not unique, is called a {\em complete intersection equation (CI equation)} of $L$.
\end{definition}

As in the case of a WI equation, a CI equation of $L$ provides a natural trivialization of the normal bundle $\nu(L)$.

\begin{definition}\label{seifert}
Let $L$ be a $k$-dimensional submanifold of $M$. A $(k+1)$-dimensional submanifold $S\subset M$ with boundary, not necessarily compact, is a {\em Seifert manifold} for $L$ if its normal bundle $\nu(S)$ is trivial and $L = \partial S$. If $M=\EE^n$ and $n-k=2$, this is just the standard notion of a {\em Seifert surface} for the usual knots and links in $\EE^n$~\cite{Ro75,KW77}.
\end{definition}

The reader will find in \cite{BK96} an alternative definition of complete intersection, more in the spirit of Algebraic Geometry. In fact, the two preceding concepts are strongly related (see~\cite[Theorem 1.12]{BK96}).

\begin{theorem}[Bochnak and Kucharz]\label{seif-int}
A submanifold $L\subset M$ of dimension $k<n$ is a complete intersection if and only if it admits a Seifert manifold.
\end{theorem}

\begin{remark}\label{R:int}
Let us denote by $H^\infty_*(M;\ZZ)$ the homology group of locally finite (but possibly infinite) simplicial chains of $M$. A Seifert manifold $S$ for $L\subset M$, having trivial normal bundle, inherits a well defined orientation from the orientation of $M$. Accordingly, if we denote by $\epsilon$ the induced orientation on $L=\partial S$, we see that $0=[L^\epsilon] \in H^{\infty}_k(M,\ZZ)$ for any submanifold $L \subset M$ which is a complete intersection. But the converse implication is not true; there are embeddings of exotic spheres in Euclidean spaces which are not CI although the corresponding homology class vanishes (see Ref. \cite{BK96}). 
\end{remark}

A SI submanifold $L\subset M$ is obviously WI and CI, so it is natural to ask whether the converse implication holds: 

\begin{question}\label{conjci}
Is a submanifold SI if and only if it is WI and CI?
\end{question}

The main difficulty of this problem is that the critical points of a CI equation $\Xi$ are not isolated in general, so it is not evident how to get rid of them in order to transform $\Xi$ into a submersion. In the forthcoming sections we will address this problem in most cases and we shall study submanifolds (e.g. knots in $\EE^3$) that are WI and CI but not SI.

On the other hand, there are examples showing that the two properties of being WI and CI are independent in any codimension.

\begin{example}\label{exsphere}
(1) The standard sphere $\mathbb{S}^k \subset \EE^{k+1}$ is defined by the equation $\Xi(x):=|x|^2-1$, which is obviously CI. But this sphere is not WI because it bounds a closed disk and so the restriction to the disk of a WI equation $\Phi: \EE^{k+1}\to \EE$ would have a critical point, which is a contradiction.

We can increase the codimension of these examples by embedding $\EE^{k+1}$ into $\EE^n=\EE^{k+1}\times \EE^{n-k-1}$, the corresponding embedding of $\e^k$ still being CI but not WI for $k\notin\{1,3,7\}$ because $\mathbb{S}^k$ is not parallelizable~\cite{BM58} (compare with Lemma~\ref{L:easy}).

(2) Consider $M = (\EE^2 \backslash \{0\}) \times \EE^{n-2}$ endowed with the restriction of the canonical system of coordinates of $\EE^n$. If $P \subset \EE^n$ is the hyperplane defined by the equation $\{x_1 = 0\}$, then  $\Phi(x) = x_1$ is a strong equation for $L = M \cap P$ in $M$ and consequently the half-plane $L^+ \subset L$ given by $\{x_2 > 0\}$ is WI but neither CI nor SI because its fundamental homology class is no longer trivial in $M$.

Similar examples in higher codimensions can be obtained by embedding $M$ in products $M \times N$.

\end{example}

\section{Codimension $1$ integrable embeddings}\label{1}

Let $L \subset M$ be a codimension one, not necessarily compact nor connected, embedded submanifold of an open $n$-dimensional manifold $M$; its normal bundle $\nu(L)$
is a rank one bundle which is orientable and thus trivial. In this section we shall show that under these assumptions the WI and SI characters of $L$ and their relationship with complete intersection (cf. Question~\ref{conjci}), can be completely understood using elementary techniques. Let us first describe a preliminary construction.

\begin{const}\label{const.tilings}

Denote by $\hat M_L$ the manifold obtained by cutting $M$ along $L$; it is a manifold diffeomorphic to the complement of an open trivial tubular neighborhood of $L$. Any of its components $V_j$ is a manifold with boundary and we denote by $J_j: V_j \to M$ the natural inclusion of $V_j$ into $M$. We make the following observations:

(1) The inclusion $J_j: V_j \to M$ is injective unless there exist two boundary components of $\partial V_j$ which are mapped onto the same component $L_i$ of $L$. This means that $L_i$ is contained in the interior of $J_j(V_j)$ and consequently there exists a loop $\alpha$ contained in the interior of $J_j(V_j)$ and meeting $L_i$, and thus also $L$, in exactly one point. Then by Poincar\'e-Lefschetz duality, we see that $0 \neq [L^\epsilon] \in H^\infty_{n-1}(M;\ZZ)$ for any orientation $\epsilon$ of $L$.

(2) Conversely, if $0 = [L^\epsilon] \in H^\infty_{n-1}(M;\ZZ)$ for some orientation $\epsilon$, the family $\mathcal{T}_L= \{V_j\}_j$ of components of $\hat M_L$ defines a tiling of $M$. Two tiles are said to be \textit{contiguous} if they intersect along some component of $L$, furthermore we can assume that the tiles are indexed by $\{0,1,2,\dots,r\}$ or $\mathbb{N}$ in such a way that $V^{[r]}:= \cup_{j=1}^r V_j$ is connected for any $r$.

Finally, we represent the dual graph $\Gamma$ of this tiling as follows:

\noindent (a) Its set of vertices is a countable family of base-points $x_j \in \text{int}(V_j)$.

\noindent (b) Any component $L_i$ of $L$ defines an edge $\gamma_{jk}$ whose endpoints are the base-points $x_j$ and $x_k$ of the two contiguous tiles $V_{j(i)}$ and $V_{k(i)}$ which intersect along $L_i$. It can be represented by an embedded path which cuts $L_i$ in exactly one point with the algebraic intersection $\gamma_{jk} \wedge L_i = \pm 1$, depending on the orientations of $\gamma_{jk}$ and $L_i$.

(3) Fix an orientation $\kappa$ of $M$ and let $\kappa_j$ be its restriction to $V_j$. We will say that $\hat M_L$ admits an \textit{alternating orientation} if there exists for each $j$ an orientation $\eta_j= \pm \kappa_j$ of $V_j$ such that any component $L_i$ of $L$ is the intersection of two contiguous tiles $V_{j(i)}$ and $V_{k(i)}$ whose orientations $\eta_j$ and $\eta_k$ are opposite. Note that the latter condition implies that the two tiles $V_{j(i)}$ and $V_{k(i)}$ are distinct, that all inclusion maps $J_j$ are one-to-one, that $\hat M_L$ defines a tiling $\mathcal{T}_L$ of $M$ and that description (2) above applies.
\end{const}

Now we get a simple characterization of complete intersection submanifolds in codimension one.

\begin{proposition}\label{thm.CI} Let $L$ be a codimension one submanifold of the open $n$-manifold $M$, then the following conditions are equivalent:
\begin{enumerate}

\item $\hat M_L$ admits an alternating orientation.
\item $L$ is a complete intersection.
\item There exists an orientation $\epsilon$ of $L$ such that $0 = [L^\epsilon] \in H^\infty_{n-1}(M;\ZZ)$.
\end{enumerate}

\end{proposition}

\begin{proof}We already noticed in Remark \ref{R:int} that $2\Rightarrow 3$, so it remains to prove the two other implications.

$1 \Rightarrow 2$: We observed in Construction~\ref{const.tilings} that if condition $1$ holds, then $\hat M_L$ defines a tiling of $M$ and any component $L_i$ of $L$ is the intersection of two distinct contiguous tiles $V_{j(i)}$ and $V_{k(i)}$. Then for any closed tubular neighborhood $N_i$ of $L_i$, there exists a trivialization $\phi_i: N_i \to [-1,+1]$ such that $\phi_i(L_i) = 0$ and $\phi_i$ has the same sign as the orientations $\eta_{j(i)}$ and $\eta_{k(i)}$ when restricted to $V_{j(i)}$ and $V_{k(i)}$ respectively. Fitting together all these maps $\phi_i$, we obtain a map $\phi$ on a tubular neighborhood $N$ of $L$ which extends continuously by $\pm 1$ on each component $V_j$. Smoothing it out we get a CI equation $\Xi$ of $L$.

$3 \Rightarrow 1$: According to paragraph $(2)$ in Construction~\ref{const.tilings}, condition $3$ implies that $\hat M_L$ defines a tiling of $M$. We consider its dual graph $\Gamma$, choose a maximal tree $\Lambda \subset \Gamma$ and define a sequence $\{\eta_j\}_j$ of orientations of the components $V_j$ by setting $\eta_0 = +1$ and $\eta_j \cdot \eta_k = -1$ for any edge $\gamma_{jk} \subset \Lambda$. We claim that this procedure defines an alternating orientation for $\hat M_L$. Indeed, the claim holds by definition if $\Lambda = \Gamma$. In the general case, take two contiguous tiles $V_j$ and $V_k$ and let $\alpha$ be the unique simplicial path joining $x_j$ to $x_k$ in $\Lambda$; it is a sum of $p$ edges and the result is a consequence of the two following observations:
\medskip

\noindent (a) $p$ must be odd because otherwise, the closed loop $\beta = \alpha + \gamma_{jk}$ would be the sum of an odd number of edges implying that the algebraic intersection $\beta \wedge L$ is not equal to zero for any orientation of $L$, thus contradicting condition 3 by Poincar\'e-Lefschetz duality.

\noindent (b) Each edge contained in $\alpha$ represents a change of sign for the corresponding orientations, so $p$ being odd, $\eta_j$ and $\eta_k$ have opposite signs.
\end{proof}

Another key idea to characterize codimension one WI and SI submanifolds is that the critical points of a CI equation of $L$ are generically isolated, and hence they can be eliminated if no component of $\hat M_L$ is bounded. Now we are ready to prove the main theorem of this section, but first let us introduce the following preliminary lemma.

\begin{lemma}\label{lem.WI} Let $L \subset M$ be a codimension one submanifold of the open $n$-dimensional manifold $M$. Then if $L$ is WI, the manifold $\hat M_L$ is open.

\end{lemma}

\begin{proof} Let $\Phi_j$ be the restriction to a component $V_j$ of $\hat M_L$ of a weak equation $\Phi$ of $L$. Then $\Phi_j$ is equal to $0$ on $\partial V_j$ and if $V_j$ is compact, it attains a maximum or a minimum in $\text{int}(V_j)$ thus providing a singular point for $\Phi_j$ and contradicting the fact that $\Phi_j$ is a submersion.
\end{proof}

\begin{theorem} \label{T:Sect1} Let $L$ be a codimension one submanifold of $M$. The following conditions are equivalent:
\begin{enumerate}
\item $L$ is SI.
\item $L$ is WI and $0 = [L^\epsilon] \in H^\infty_{n-1}(M;\ZZ)$ for some orientation $\epsilon$ of $L$.
\item $L$ is CI and WI.
\item $L$ is CI and $\hat M_L$ is open.
\end{enumerate}
\noindent Moreover, $L$ is WI if and only if $\hat M_L$ is open.
\end{theorem}

\begin{proof} (1) First notice that implications $1 \Rightarrow 2 \Rightarrow 3\Rightarrow 4$ are straightforward due to Proposition~\ref{thm.CI} and Lemma~\ref{lem.WI}, so let us prove the remaining one $4 \Rightarrow 1$.

Take a CI equation $\Xi: M \to \EE$ of $L$ and let $\Xi_j$ be its restriction to a component $ V_j$ of $\hat M_L$. Since Morse functions are dense in $C^\infty(V_j)$ \cite{Hi76} in the strong topology, we approximate $\Xi_j$ by a Morse function $\Phi_j:V_j\to\EE$ which coincides with $\Xi_j$ in a neighborhood of $\partial V_j$ and does not vanish in $\text{int}(V_j)$. By definition, the critical set of $\Phi_j$ consists of countably many isolated points $\{p_k\}_{k\in \mathbb N}$ contained in
$\text{int}(V_j)$. Assuming that $V_j$ is not compact, there exists for each $k$ a curve $\alpha_k$ properly embedded in $\text{int}(V_j)$, which joins $p_k$ to infinity. Moreover, we can choose
these curves so that they admit pairwise disjoint tubular
neighborhoods $A_k$. Now it is a standard observation that $A_k$
is diffeomorphic to $A_k \backslash \alpha_k$ and therefore $V_j$ is
diffeomorphic to $V'_j:= V_j\backslash (\bigcup_k \alpha_k)$. The
restriction $\Phi'_j$ of $\Phi_j$ to $V'_j$ is a submersion which extends the restriction of $\Xi_j$ to a neighborhood of $\partial V_j$ and does not vanish on $\text{int}(V'_j)$.
Fitting all these maps $\Phi'_j$ together, we finally obtain a strong equation $\Phi'$ of $L$ in $M':= \cup_j V'_j$ which is diffeomorphic to $M$.

(2) To achieve the characterization of WI embeddings, it remains to prove the converse of Lemma \ref{lem.WI}. So let $N_L$ be a closed tubular neighborhood of $L$ in $M$. The boundary $\partial N_L$ of $N_L$ is the union of two copies $L'$ and $L''$ of $L$ and $N_L$ is a Seifert manifold for $F = L' \cup L''$; thus $F$ is CI by Theorem \ref{seif-int}. Now denote by $N_i$ the connected component of $N_L$ containing the component $L_i$. Any connected component of $\hat M_F$ is either diffeomorphic to some $V_j$ or equal to some $N_i$, then assuming that $\hat M_L$ is open, there will be two possibilities:
\medskip

\noindent (a) If all components $L_i$ of $L$ are open, all $N_i$ and therefore also all components of $\hat M_F$ will be open as well. We conclude that $F$ is SI by the discussion in paragraph (1), and hence $L'$ contained in $F$ is WI. The same is true for $L$ because it is isotopic to $L'$ in $M$.

\noindent (b) If some component $L_i$ is compact, so is the corresponding component $N_i$ of $N_L$ and there exists a component $V_{j(i)}$ of $\hat M_L$ containing $L''_i $. As $V_{j(i)}$ is not compact, we join some point $p_i \in L''_i$ to infinity in $V_{j(i)}$ by means of an embedded path $\alpha_i$ and, proceeding in the same way as in paragraph (1), we remove all these paths $\alpha_i$ defining a manifold $M'$ which is diffeomorphic to $M$. The submanifold $F' = F \cap M'$ is still CI and the manifold $\hat M'_{F'}$ is open; thus as in $(1)$, we conclude that $ F'$ is SI, from which we deduce that $L' \subset F'$ and therefore also $L$ is WI.
\end{proof}

Note that Theorem \ref{T:Sect1} answers Question~\ref{conjci} in codimension $1$. See also Example~\ref{exsphere} of Section~\ref{S:pre} for examples of codimension one submanifolds which are WI and not CI, or CI but not SI.

\begin{corollary}\label{thm.cod1} If the first Betti number $\beta_1(M):=\text{rank}\,\,H^1(M;\ZZ)$ of $M$ is zero, any codimension one submanifold $L \subset M$ is a complete intersection. Moreover the following conditions are equivalent:

\begin{enumerate}
\item $L$ is open.
\item $L$ is WI.
\item $L$ is SI.
\end{enumerate}
\end{corollary}
\begin{proof}Recall that by Poincar\'e-Lefschetz duality, $H^1(M,\ZZ)$ is
isomorphic to the homology group $H^\infty_{n-1}(M,\mathbb{Z})$ of
locally finite singular chains. Thus assuming that
$\beta_1(M)=0$, we obtain that $0=[L^\epsilon]\in H^\infty_{n-1}(M,\mathbb{Z})$ for any orientation $\epsilon$ of $L$ and hence $L$ is a complete intersection by Proposition \ref{thm.CI}.

Next notice that the same observation holds for any component $L_i$ of $L$, thus $L$ is open if and only if $\hat M_L$ is open. Then the three conditions are equivalent by Theorem \ref{T:Sect1}.
\end{proof}

Here is another consequence which is not at all obvious a priori.

\begin{corollary}\label{ccc}
Any open codimension one submanifold $L\subset \EE^n$ is parallelizable and can be immersed into $\EE^{n-1}$.
\end{corollary}
\begin{proof}Indeed $L$ is SI by Corollary \ref{thm.cod1} and parallelizable by Lemma \ref{L:easy}. The second claim follows from the fact that an open $k$-dimensional manifold immerses into $\EE^k$ if and only if it is parallelizable \cite{Hi61}.
\end{proof}

\begin{appl} Using Corollary \ref{thm.cod1} we easily recover the theorem of Bouma and Hector \cite{HB83}, which states that any open surface (properly) embedded in $\EE^3$ is SI.
\end{appl}

\section{Integrable embeddings: codimension $m\geq 2$}\label{s.general}

In this section we consider submanifolds $L\subset M$ embedded in codimension $m\geq 2$. In this case, the integrability problem is much more difficult because the dimension of the critical set of a map from $M$ to $\EE^m$ is generically positive \cite{Th55}, which makes it complicated to get rid of all critical points. The main trick that we shall use is to reduce the integrability problem to a question in homotopy theory. As in the earlier
paper by Miyoshi \cite{Mi95}, an essential role will be
played by a relative version \cite{Ha71} of the so-called Phillips--Gromov h-principle.

Denote by $Sub(M,\EE^m)$ the space of submersions from $M$ to $\EE^m$ and by
$Max[T(M),T(\EE^m)]$ the space of bundle morphisms from the tangent
bundle $T(M)$ to the tangent bundle $T(\EE^m)$ which are of maximal
rank $m$ on each fiber. Then by Phillips-Gromov's
theory, the differential map
$$
\dd:\,Sub(M,\EE^m) \to Max[T(M),T(\EE^m)]
$$
is a weak homotopy equivalence \cite{Ph67,Gr69}.

\smallskip

In fact, we will need a relative version of this theorem, so let us
introduce some notation. Let $U\subset M$ be a codimension $0$ closed
submanifold with smooth boundary and let $\varphi:U \to \EE^m$ be a fixed smooth submersion.

\begin{notation} We denote by

(1) $Sub_{\varphi}(M,\EE^m)\subset Sub(M,\EE^m)$ the space
of submersions whose restriction to $U$ is equal to $\varphi$.

(2) $Max_{\dd\varphi}[T(M),T(\EE^m)]\subset Max[T(M),T(\EE^m)]$ the space
of maximum rank morphisms whose restriction to $T(U)$ is equal to
$\dd\varphi$.

\end{notation}

Now the relative version of the Phillips-Gromov h-principle for submersions can
be stated as follows \cite{Ha71}.

\begin{theorem}[Philips-Gromov, relative version]\label{Gromov}
Let $U \subset M$ be a closed codimension $0$ submanifold of an open $n$-manifold $M$. If $W:= M \backslash \text{int}(U)$ is open, then for any submersion $\varphi: U \to \EE^m$, the differential map
$$
\dd: Sub_{\varphi}(M,\EE^m)\to Max_{\dd \varphi}[T(M),T(\EE^m)]
$$
is a weak homotopy equivalence.
\end{theorem}

Let us observe that the above condition on $W$ in Theorem \ref{Gromov} is analogous to the assumption that $\hat M_L$ is open in Theorem \ref{T:Sect1}. The common reason for these two requirements is that we want to ``push critical points to infinity'' for which we need the ambient manifold to be non-compact.

In order to fix our notations we recall the usual constructions of tubular neighborhoods and different associated trivializations.

\begin{const}\label{const.trivial} Let $L\subset M$ be a
codimension $m\geq 2$ submanifold with trivial normal bundle $\nu(L)$.

(1) Fix a complete Riemannian metric $g$ on $M$. It is a standard fact that there exists a smooth function $f$ on $L$ such that, if we denote by $E\subset \nu(L)$ the disk-subbundle whose fiber $E_x$ over $x\in L$ is the closed disk of radius $f(x)$ in $\nu_x(L)$, the exponential map $exp_g:\nu(L)\to M$ defines a diffeomorphism of $E$ onto a closed tubular neighborhood $N_L$ of $L$ in $M$. If $L$ is compact we can assume without trouble that $f$ is constant.

(2) A \textit{trivialization} of $\nu(L)$ is a bundle map $\sigma: \nu(L) \to \RR^m$ of maximal rank. Normalizing $\sigma$ by means of $f$, we can assume that $\sigma(E_x)$ is the closed unit disk $D^m\subset \RR^m$ for any $x\in L$. Thus denoting by $\pi:\nu(L)\to L$ the bundle projection, the tubular neighborhood $N_L$ is endowed with a natural product structure defined by the map

$$(q,\theta):=(\pi\circ exp_g^{-1},\sigma\circ exp_g^{-1}):N_L\to L\times D^m\,,$$
such that $L=\theta^{-1}(0)$ and any fiber of $\theta$ is a submanifold $L'$ isotopic to $L$. We call $\theta = \sigma \circ exp_g^{-1}$ the {\it differentiable trivialization of $N_L$ associated to the infinitesimal trivialization $\sigma$ of $\nu(L)$}.

(3) Set $V:=M\backslash\text{int}(N_L)$, the restriction $\phi:=\theta|_{\partial V}$ maps $\partial V=\partial N_L$ onto the sphere $\e^{m-1}=\partial D^m$, thus defining what we will call the {\em associated trivialization of $\partial N_L$}. Note that we can recover $\theta$ and $\sigma$ from $\phi$.

(4) Finally recall that a weak equation $\Phi$ of $L$ provides a canonical infinitesimal trivialization $\sigma=\dd \Phi|_L$ of $\nu(L)$, which in turn gives a tubular neighborhood $N_L$ together with trivializations $\theta$ and $\phi$ as described previously. All these data will be said to be \textit{associated to} $\Phi$. Similar observations and notations hold for a CI equation $\Xi$.
\end{const}

\begin{remark}
It is worth noting that any extension of $\phi: \partial V \to \mathbb{S}^{m-1}$ over $V$ provides a CI equation of $L$. In particular, we shall use this fact when $L$ has codimension 2, where $\phi$ can be interpreted as a $1$-dimensional cohomology class of $L$ which can be studied by algebraic arguments (see Section \ref{s.cod2-2}).
\end{remark}

The next theorem and its corollary present the main results of this section. They imply Theorem A and Corollary B of Section~\ref{S.mostrelevant} when restricting to compact connected manifolds.

\begin{theorem}\label{thm.main}
Let $L \subset M$ be a submanifold of codimension $m\geq 2$. The
following statements are equivalent:
\begin{enumerate}
\item $L$ is WI.
\item The normal bundle $\nu(L)$ admits an extension $\tilde\nu(L)$
over $M$ which is a trivial bundle.
\end{enumerate}
\end{theorem}

\begin{proof}
$1\Rightarrow 2$: This is just Lemma \ref{L:easy}.

$2\Rightarrow 1$: Let us choose a trivialization $\widetilde \sigma : \tilde\nu(L) \to
\RR^m$ of $\tilde\nu(L)$ and use the induced trivialization
$\sigma$ of $\nu(L)$ to construct a tubular neighborhood
$N_L$ and a map $\theta:N_L \to D^m$ as in Construction
\ref{const.trivial}. Composing with a trivialization $q: T(D^m) \to \RR^m$ of $T(D^m)$, we define a map $d \theta_*: T(N_L) \to \RR^m$ of maximal rank $m$, by setting $d \theta_* := q \circ d\theta$. The result now follows in two steps:

(1) For any $x\in M$ consider the orthogonal projection
$\pi_x:T_x(M)\to\tilde\nu_x(L)$ defined by the Riemannian metric
$g$. The linear map
$$
\Theta_* : T(M) \to \RR^m
$$
given by $\Theta_*(x,v) = (\tilde\sigma\circ\pi_x)(v)$ for $v \in
T_x(M)$, is a bundle map of maximal rank which coincides with
$d\theta_*$ on $T(N_L)$. Next a retraction of $N_L$ onto $L$ yields a homotopy of maps of maximal rank between $d \theta_*$ and $d \theta$. Finally the homotopy extension theorem \cite{Sp66} gives  a map of maximal rank
$$
\Theta: T(M) \to T(\EE^m)
$$
which is homotopic to $\Theta_*$ and coincides with $d\theta$ on $T(N_L)$.

(2) Since $m \geq 2$, the manifold $M \backslash L$ is connected and hence $V:= M \backslash \text{int}(N_L)$ is connected and open. Therefore the relative h-principle (see Theorem \ref{Gromov}) applies with $(U,\varphi) = (N_L,\theta)$. It follows that $\Theta$ is homotopic to the differential of a submersion $\Phi \in Sub_{\theta}(M,\EE^m)$ which is a weak equation of $L$ due to the fact that its restriction to $N_L$ coincides with $\theta$.
\end{proof}

Let us observe that we do not have any control on the image by $\Phi$ of the complement of $N_L$, that is why we cannot conclude that $L$ is SI. Note also that in the previous proof we have used a special trivialization of the tubular
neighborhood $N_L$, namely one which is obtained from
a trivialization of the extended bundle $\tilde\nu(L)$.

\begin{corollary}\label{cor.main} Let $L\subset M$ be a submanifold of codimension $m\geq 2$. If $M$ is contractible, the following conditions are equivalent:
\begin{enumerate}
\item $L$ is WI.
\item The normal bundle $\nu(L)$ extends over $M$.
\item The tangent bundle $T(L)$ extends over $M$.
\end{enumerate}
Moreover any submanifold $L$ which satisfies one of these conditions is parallelizable and has trivial normal bundle.
\end{corollary}

\begin{proof} As $M$ is contractible, any extension of $\nu(L)$ over $M$ is trivial, and hence using Theorem \ref{thm.main} we see that statements 1 and 2 are equivalent. Note also that the extensions $\tilde \nu(L)$ and $\tilde T(L)$ of $\nu(L)$
and $T(L)$ are complementary subbundles of $T(M)$, so statements
2 and 3 are also equivalent. Finally, Lemma \ref{L:easy} shows that $L$ is parallelizable and has trivial normal bundle.
\end{proof}

\begin{remark}It is worth noting that the two conditions in Theorem \ref{thm.main} are no longer equivalent in codimension $1$. This is because, in this case, the set $M\backslash \text{int}(N_L)$ may possess a compact piece, thus prohibiting the use of Theorem~\ref{Gromov}.

Indeed the standard embedding of the torus $\mathbb{T}^2$ in $\EE^3$ is a leaf of a Reeb foliation of $\EE^3$ and hence its normal bundle $\nu(\mathbb{T}^2)$ is extended over $\EE^3$ by the normal bundle of this foliation, which is of course trivial. Being compact, this torus cannot be WI according to Corollary \ref{thm.cod1}.

On the contrary, if $L$ is of codimension 1 but open, then $M\backslash \text{int}(N_L)$ is open as well, Phillips-Gromov theorem applies and Theorem \ref{thm.main} still holds true. Moreover the normal bundle of $L$ has rank one and extends over $M$ as a trivial bundle, thus we conclude that $L$ is WI. This provides another explanation for the condition imposed to $L$ in Theorem \ref{T:Sect1}, as well as an alternative approach to the proof of this theorem.
\end{remark}

Let us conclude this section with an application of Theorem~\ref{thm.main} and Corollary~\ref{cor.main} to embeddings of $1$-dimensional submanifolds. In particular, we shall recover two theorems by Watanabe~\cite{Wa93} and Miyoshi~\cite{Mi95}.

\begin{appl}\label{appl.low}

(1) Any $1$-dimensional manifold $L$ embedded in
a contractible $n$-manifold $M$, $n\geq 3$, has trivial normal bundle and is WI. Indeed, consider a non-vanishing smooth vector field $\hat X$ on $L$ and extend it
all over $M$ as a vector field $X$ with a priori many
singularities \cite{Wh34}. It is standard \cite{PM82} that $X$ can be approximated by a vector field $Y$ whose
singularities are all hyperbolic and hence isolated. The singularities of $Y$ can be eliminated by a procedure similar to that used in the proof of Theorem \ref{T:Sect1} for the elimination of singularities of Morse functions, thus showing
that $\hat X$ can be extended as a non-vanishing smooth vector
field $Z$ on $M$. This proves that $T(L)$ extends as a trivial subbundle
of $T(M)$; the claim follows by
Corollary \ref{cor.main}. In particular, we recover Watanabe's theorem~\cite{Wa93} when
$L$ is a link in $\EE^3$.

(2) Any $1$-dimensional manifold $L$ embedded in
an open $3$-manifold $M$ is WI. A particular case of this claim was stated by Miyoshi in~\cite{Mi95}, but his proof is not complete because there is missing the fact that the normal bundle $\nu(L)$ extends over $M$ as a trivial bundle. We can easily fill this gap using our Theorem~\ref{thm.main}. Indeed, let us show that the normal bundle $\nu(L)$ extends over $M$ as a trivial bundle. This fact is a consequence of the existence of certain contact structures, i.e. completely non-integrable $2$-plane distributions, over $M$. A classical result of Gromov asserts that there exists on any open $3$-manifold, a contact structure $\xi$ which admits a global trivialization (see \cite{Gr69}, or \cite{EM02}). Furthermore, it is a standard property in contact topology \cite[Section 14.2]{EM02} that, given $L$, there is an arbitrarily $C^0$-small diffeomorphism $\Upsilon:M\to M$ such that $\Upsilon(L)$ is transverse to $\xi$. Thus we can assume that $\nu(L)$ identifies with the restriction of $\xi$ to $L$ and therefore extends over $M$ by $\xi$, which is trivial by assumption. Remember that the normal bundle of a $1$-dimensional manifold properly embedded in an open $3$-manifold is always trivial if orientable.

In dimension $n>3$ it is necessary to assume that the manifold $M$ is parallelizable, in which case, using again contact topology, it is not difficult to prove that any $1$-dimensional submanifold of $M$ is WI provided that $n$ is odd.
\end{appl}

\section{Integrability and complete intersection in codimension $m\geq 2$}\label{s.cod2}

As usual, let $L\subset M$ be a manifold of dimension $k$ embedded in the $n$-dimensional manifold $M$ and set $m:=n-k$. Here we want to characterize strongly integrable submanifolds of codimension $m\geq 2$. This will be done in terms of weak integrability and complete intersection (recall Question~\ref{conjci}). In particular, we shall emphasize the essential role played by the fundamental class $[L^\epsilon] \in H_{n-2}^\infty(M)$ of $L$ in the case of codimension $2$ submanifolds.

\subsection{Tame complete intersection}\label{s:tame}

We noticed in Section~\ref{S:pre} that, as for WI submanifolds, the normal bundle of a CI submanifold is trivial, the main difference being that, in the CI case, this bundle does not need to extend over $M$. We take this difference into account in the following definition:

\begin{definition}\label{D:ci} A complete intersection submanifold $L \subset M$ is called a \textit{tame complete intersection (TCI)} if the following conditions hold:
\begin{enumerate}
\item The normal bundle $\nu(L)$ extends as a bundle $\tilde \nu(L)$ over $M$.
\item There exists a CI equation $\Xi$ of $L$ such that the associated trivialization $\sigma$ of $\nu(L)$ extends as a trivialization $\tilde \sigma$ of $\tilde \nu(L)$.
\end{enumerate}

\end{definition}

It is obvious that a TCI submanifold is also WI by Theorem~\ref{thm.main}, in fact the condition $2$ means that the two characters WI and CI are related by the fact that they provide a common trivialization for the normal bundle (compare with the proof of Theorem~\ref{thm.main}).

In the next theorem, we describe a procedure which allows us to eliminate the critical points of a tame CI equation $\Xi$ without changing the set $\Xi^{-1}(0)$, thus showing that a TCI submanifold is also SI, a fact already included in Theorem~\ref{T:Sect1} for codimension $1$ embeddings. As in Section~\ref{s.general}, the main tool is the Phillips-Gromov h-principle, and we make use of the notations introduced in Construction~\ref{const.trivial} without further mention.

\begin{theorem}\label{thm.main2}
Let $L \subset M$ be a codimension $m\geq 2$ submanifold. Then $L$ is SI if and only if it is TCI.
\end{theorem}
\begin{proof}
Indeed, a strong equation $\Phi$ of $L$ is also obviously a tame CI equation of $L$, so it remains to show the converse implication.

Let $\Xi$ be a tame CI equation of $L$ and let $\tilde \sigma$ be the associated trivialization of the extension $\tilde \nu(L)$ of $\nu(L)$. We proceed in two steps as in the proof of Theorem~\ref{thm.main}.

(1) Using a Riemannian metric $g$ on $M$, we define again a field of orthogonal projections $\pi_x:T_x(M)\to \tilde\nu_x(M)$, for $x\in M$. Next we use the tame CI equation $\Xi:M\to \EE^m$, which is a priori not a submersion, to construct an epimorphism of bundles $\Theta:T(M)\to T(\EE^m)=\EE^m\times \RR^m$ by setting

$$
\Theta(x,v):= [\Xi(x),(\tilde\sigma\circ\pi_x)(v)]\,,
$$
for $x\in M$ and $v\in T_x(M)$. Let $N_L$ and $\theta$ be a tubular neighborhood of $L$ and its corresponding differentiable trivialization, both associated to $\Xi$. After possibly changing the CI equation $\Xi$ in $N_L$, there is no loss of generality in assuming that $\theta= \Xi|_{N_L}$ and $N_L=\Xi^{-1}(D^m)$, and the following relations hold by construction:
$$
\Theta|_{T(N_L)}=\dd\theta ~~\text{and}~~ \Theta[T(M\backslash\text{int}(N_L))]\subset T[\EE^m\backslash \text{int}(D^m)]\,.
$$

(2) Next choose a second closed tubular neighborhood $N'_L \subset \text{int}(N_L)$ associated to $\Xi$ and such that $\Xi(N'_L)$ is a closed disk $D'$ contained in the interior of $D^m = \Xi(N_L)$; we set $M':=M\backslash N'_L$ and $U:= M' \cap N_L$. As $M'$ and $M'\backslash\text{int}(U)$ are open, we can apply Theorem~\ref{Gromov}  to the pair $M'\supset U$ with $\varphi=\theta|_U$ to conclude that $\Theta|_{T(M')}$ is homotopic to the differential of a submersion $\Phi': M' \to \EE^m\backslash D'$ such that $\Phi'|_U=\theta|_U$. Finally we glue $\Phi'$ and $\theta$ together along $U$ by setting:
$$
\Phi(x) = \begin{cases} \theta(x) & \text{if $x \in
N_L$\,,}\\
\Phi'(x) &\text{if $x \in M'$}\,.
\end{cases}
$$
The submersion $\Phi: M \to \EE^m$ we obtain this way is obviously a strong equation of $L$, thus finishing the proof of the theorem.
\end{proof}

\begin{remark}
Question~\ref{conjci} remains open in general because, if $L$ is both WI and CI, it is not easy to ascertain that the two equations can be made compatible so as to admit a common associated trivialization of the normal bundle. In fact, in Section~\ref{S:miyoshi} we shall show some examples of $k$-manifolds in $\EE^{2k+1}$ which are WI and CI but not SI.
\end{remark}

\subsection{Codimension 2 submanifolds}\label{s.cod2-2}

Next we focus on a particular case in which Question~\ref{conjci} can be partially answered: submanifolds $L\subset M$ of codimension $m=2$.

The important and specific property of codimension 2 is that, as in the case of codimension one, the condition $0=[L^\epsilon]\in H^\infty_{n-2}$ is enough to prove the existence of a Seifert manifold for $L$ (see \cite{BK96}) and hence to conclude that $L$ is a complete intersection (see Theorem \ref{seif-int} and  Remark \ref{R:int}). The techniques used here are ``standard'', usually employed in the construction of Seifert manifolds for higher dimensional knots \cite{Ro75,BK96} and already used by Miyoshi in his study of SI links in open $3$-manifolds.

Adapting the notations introduced in Construction \ref{const.trivial}, let $\theta:N_L\to D^2$ and $\phi: \partial V = \partial N_L \to \partial D^2=\e^1$ be differentiable trivializations associated to a trivialization of $\nu(L)$. Note that any fibre $L'$ of $\phi$ is isotopic to $L$ in $N_L$, so it inherits an orientation $\epsilon$. The following lemma and proposition characterize complete intersection in a way similar to the codimension $1$ case (see Proposition \ref{thm.CI}).

\begin{lemma}\label{L:h1}
Let $L \subset M$ be a codimension $2$ submanifold with trivial normal bundle. If $H^1(L;\ZZ)=0$ then:
\begin{enumerate}
\item The fibers of any two associated trivializations of $\partial N_L$ are isotopic in $\partial N_L$.
\item If $0=[L^\epsilon] \in H^\infty_{n-2}(M,\ZZ)$ for some orientation $\epsilon$, the typical fiber $L'$ of any associated trivialization $\phi$ satisfies $0=[L']\in H^\infty_{n-2}(V;\ZZ)$.
\end{enumerate} 
\end{lemma}
\begin{proof}
(1) The triviality of the normal bundle of $L$ implies that $\partial N_L=\e^1\times L$, so the typical fiber $L'$ of a trivialization $\phi$ is the graph in $\partial N_L$ of a map $\vartheta:L\to \e^1$. Recall that we can identify the group $[L,\e^1]$ of homotopy classes of maps from $L$ to $\e^1$ with the cohomology group $H^1(L,\ZZ)$  by means of the Eilenberg-MacLane isomorphism $E:[L,\e^1]\to H^1(L;\ZZ)$ (see~\cite{Sp66}), thus implying that the homotopy class of $\vartheta$ does not depend on the trivialization because $H^1(L;\ZZ)=0$. By a standard argument, we conclude that the fibers $L'$ of any two associated trivializations $\phi$ are isotopic in $\partial N_L$, thus proving the first claim.

(2) As $L$ has codimension $2$, if we assume that $0=[L^\epsilon] \in H^\infty_{n-2}(M,\ZZ)$, it follows~\cite[Theorem 1.11]{BK96} that $L$ admits a CI equation $\Xi:M\to\EE^2$. Calling $\phi_\Xi:\partial N_L\to \e^1$ its associated trivialization, the typical fiber $F$ of $\phi_\Xi$, endowed with the corresponding orientation, defines the zero class in $H^\infty_{n-2}(V;\ZZ)$, i.e. $0=[F^\epsilon]\in H^\infty_{n-2}(V;\ZZ)$, by Theorem~\ref{seif-int} and Remark~\ref{R:int}. This fact and claim $1$ prove claim $2$.
\end{proof}

\begin{proposition}\label{l.extension}
Let $L \subset M$ be a submanifold as in Lemma~\ref{L:h1}. Then the following conditions are equivalent:
\begin{enumerate}
\item There exists an orientation $\epsilon$ of $L$ such that $0=[L^\epsilon] \in H^\infty_{n-2}(M,\ZZ)$.
\item Any associated trivialization $\phi:\partial N_L=\partial V\to \e^1$ extends as a map $\psi: V\to \e^1$.
\item $L$ is a complete intersection and admits a Seifert manifold.
\end{enumerate}
\end{proposition}

\begin{proof}
$3 \Rightarrow 1$: It follows from Remark~\ref{R:int}.

$1\Rightarrow 2$: As before we can identify the group $[V,\e^1]$ of homotopy classes of maps from $V$ to $\e^1$ with the cohomology group $H^1(V,\ZZ)$  using the Eilenberg-MacLane isomorphism $E:[V,\e^1]\to H^1(V;\ZZ)$. Now consider the commutative diagram:

$$
\xymatrix{
[V,\e^1] \ar[r]^R \ar[d]^E & [\partial V,\e^1]\ar[d]^E &  \\
H^1(V;\ZZ) \ar[d]^D \ar[r]^j & H^1(\partial V;\ZZ) \ar[r]^{d} \ar[d]^D & H^2(V,\partial V;\ZZ) \ar[d]^D\\
H^\infty_{n-1}(V,\partial V;\ZZ) \ar[r]^{\delta} &
H_{n-2}(\partial V;\ZZ) \ar[r]^i & H^\infty_{n-2}(V;\ZZ)
\\}
$$
where $R$ is the restriction map and the two last horizontal sequences
are homology and cohomology exact sequences related by vertical
Poincar\'e-Lefschetz duality isomorphisms $D$ \cite{Sp66}.

By Lemma~\ref{L:h1} any fiber $L'\subset \partial V$ of $\phi$ defines the trivial homology class in $H^\infty_{n-2}(V;\ZZ)$ when endowed with the corresponding orientation. Furthermore, the cohomology class $[\phi]\in H^1(\partial V;\ZZ)$ is the Poincar\'e-Lefschetz dual of $[L']\in H_{n-2}(\partial V;\ZZ)$.

As $[L']$ is zero in $H^\infty_{n-2}(V;\ZZ)$, there
exists an element $\alpha \in H^1(V;\ZZ)$ such that $[\phi] =
j(\alpha)$ and the lift of $\alpha$ to $[V,\e^1]$ by $E$ is the
homotopy class of a map $\psi: V \to S^1$ whose restriction to
$\partial V$ is homotopic to $\phi$. The map $\psi$ can be taken smooth, and changing it by means of the homotopy extension theorem, we can assume that $\psi|_{\partial V}=\phi$ thus proving the claim.

$2\Rightarrow 3$: As the set $\Delta:=\EE^2\backslash \text{int}(D^2)$ retracts onto
$\partial D^2$, the homotopy extension theorem provides
a map $\psi':V\to \Delta$ homotopic to $\psi$ such that $\psi'|_{\partial V}=\psi|_{\partial V}=\theta|_{\partial N_L}$. We define a CI equation $\Xi_\theta:M\to \EE^2$ of $L$ depending on the trivialization $\theta$ by setting:

$$
\Xi_\theta(x) = \begin{cases} \theta(x) & \text{if $x \in
N_L$\,,}\\
\psi'(x) &\text{if $x \in V$}\,.
\end{cases}
$$

\end{proof}

Note that the proof of the implication $1\Rightarrow 2$ cannot be adapted to higher codimension because $\e^{m-1}=\partial D^m$ is no longer an Eilenberg-MacLane space for $m>2$.

As a consequence of the previous proposition we can fully understand the relationship between SI, WI and complete intersection for codimension $2$ submanifolds with $H^1(L;\ZZ)=0$. In particular, we tackle completely Question~\ref{conjci} in this context.

\begin{theorem}\label{T:SIandTCI}
Let $L\subset M$ be a codimension $2$ submanifold with trivial normal bundle and $H^1(L;\ZZ)=0$. The following conditions are equivalent:
\begin{enumerate}
\item $L$ is WI and a complete intersection.
\item $L$ is WI and $0=[L^\epsilon] \in H^\infty_{n-2}(M,\ZZ)$ for some orientation $\epsilon$.
\item $L$ is TCI.
\item $L$ is SI.

\end{enumerate}
\end{theorem}
\begin{proof}
First note that $1\Leftrightarrow 2$ by Proposition~\ref{l.extension} and that $3\Leftrightarrow 4$ by Theorem~\ref{thm.main2}, while $4 \Rightarrow 1$ is obvious. Finally to show that $2\Rightarrow 3$, we specialize the proof of the preceding Proposition~\ref{l.extension}. Indeed, if the trivialization $\phi$ of $\partial V = \partial N_L$ used to construct a CI equation of $L$ is associated to a WI equation $\Phi$, then the corresponding CI equation $\Xi_\theta$ is automatically tame (recall Theorem \ref{thm.main} and Definition \ref{D:ci}).
\end{proof}

Proceeding as in the proof of Corollary \ref{cor.main}, the following result is straightforward using the fact that the group $H^\infty_{n-2}(M;\ZZ)$ is zero if $M$ is contractible.

\begin{corollary}\label{cor.main2}
Let $L\subset M$ be a codimension $2$ submanifold with $H^1(L;\ZZ)=0$, and assume that $M$ is contractible. Then the following conditions are equivalent:
\begin{enumerate}
\item $L$ is WI.
\item $L$ is SI.
\item The normal bundle $\nu(L)$ extends over $M$.
\item The tangent bundle $T(L)$ extends over $M$.
\end{enumerate}
Moreover, any submanifold $L$ which satisfies one of these conditions is parallelizable and has trivial normal bundle.
\end{corollary}

\begin{remark}
Theorem~\ref{T:SIandTCI} does not cover the case of embeddings of $\e^1$ in an open $3$-manifold. This is the goal of Miyoshi's theorem~\cite{Mi95}, which unfortunately is seriously flawed as we shall explain in Section~\ref{S:miyoshi}. In the case that $M=\EE^3$, we shall prove the right analogue of Miyoshi's theorem in Section~\ref{S:miyoshi}, and in particular we shall see that no embedding of $\e^1$ in $\EE^3$ is SI.   
\end{remark}

To finish this section let us give more examples of codimension $2$ submanifolds which are complete intersections but not WI (compare with Example~\ref{exsphere}).

\begin{example} Any embedding $h:\e^{k} \to \EE^{k+2}$ of the $k$-dimensional sphere has trivial normal bundle \cite{Ma59}. Furthermore, as  $0 = [h(\mathbb{S}^k)] \in H_k^\infty(\EE^{k+2};\mathbb{Z})$, any such embedding is CI, but not WI for $k\notin\{1,3,7\}$ according to Lemma \ref{L:easy}. The case of $k=1$ was already studied in Application~\ref{appl.low} while the remaining cases $k \in\{3,7\}$ will be covered by Proposition \ref{spheres}.
\end{example}

\chapter{Integrability in Euclidean spaces}\label{Ch:Euc}

\section{Introduction}\label{Euc0}

In Chapter \ref{Ch:gen} we have analyzed weak and strong integrability of embeddings in a general manifold $M$. It is clear that if we want to obtain deeper results and eventually a complete characterization of WI or SI submanifolds, we must restrict our attention to specific ambient spaces. In this chapter we will focus on the first natural candidates to study i.e. Euclidean spaces, hence $L$ will stand for a $k$-dimensional submanifold of $\EE^n$, $n>k$, not necessarily connected nor compact.

In the case of codimension one, the submanifolds which are WI or SI in $\EE^n$ are completely classified by Corollary \ref{thm.cod1}, therefore we will assume in what follows that $n\geq k+2$. In this case, on account of Corollary \ref{cor.main}, the problem of weak integrability is equivalent to that of extending over $\EE^n$ the normal (or tangent) bundle of $L$. As we shall see, this condition reduces the problem to homotopy theory and provides several obstructions to weak integrability. These obstructions are related to the theory of immersions into Euclidean spaces widely developed during the period $1955-1970$, in particular to the Hirsch-Smale classification theorem \cite{Hirsch59, Sm59}. This connection is very fruitful because it brings to the foreground several techniques like the theory of degree, and culminates in a complete classification of WI submanifolds of $\EE^n$.

Recall that a preliminary condition for a submanifold $L\subset \EE^n$ to be WI is that it is parallelizable (see Lemma \ref{L:easy}), so in what follows we will often assume that $T(L)$ is trivial.

In Section \ref{Euc1} we review Hirsch-Smale theory and other standard results concerning immersions and embeddings of parallelizable manifolds. In Section~\ref{S:tangent} we define the \textit{tangential characteristic class} of a parallelizable $k$-submanifold and classify WI submanifolds in $\EE^n$ for $n\geq 2k+1$. In Section \ref{Euc3} we introduce the concepts of \textit{normal characteristic class} and \textit{total weak integrability} and give a complete answer to the question of WI embeddings for manifolds of dimension $k\neq 3,7$. The case of 3- and 7-manifolds embedded in $\EE^n$ is studied in Subsection \ref{Euc4}, where we present examples of parallelizable submanifolds which are not WI. In Sections~\ref{Euc5} and~\ref{S:miyoshi} we focus on strong integrability, clarifying the connection between CI and SI submanifolds (in particular for links in $\EE^3$) and giving counterexamples to Miyoshi's theorem. Finally, in Section \ref{Euc6} we provide some applications of the previous results to the theory of foliations in Euclidean spaces.

\section{Preliminary results: the Hirsch-Smale classification of immersions}\label{Euc1}

Let $L$ be a $k$-dimensional submanifold of $\EE^n$ and as usual denote by $\nu(L)$ and $T(L)$ the normal and tangent bundles of $L$ respectively. The next lemma provides some sufficient conditions in order that $\nu(L)$ be trivial, a necessary condition for weak integrability.

\begin{lemma}\label{L:trivnu}
The normal bundle of $L\subset \EE^n$ is trivial for example in the following two situations:
\begin{enumerate}
\item $L$ is a codimension two submanifold.
\item $L$ is parallelizable and $n\geq 2k$.
\end{enumerate}
\end{lemma}
\begin{proof}
Both properties are well known. Property 1 can be consulted in \cite{Th52} (see also \cite{Ma59}) and property 2 is proved in \cite{Wh40} and \cite{Ke59_2}.
\end{proof}

Next we introduce the Hirsch-Smale theory. A homotopy $g:L\times [0,1]\to M$ is called \textit{regular} if for any $t\in [0,1]$, the map $g_t$ is an immersion and $g$ induces a continuous homotopy of bundle maps $g_*: T(L) \times [0,1] \to T(M)$ such that $(g_*)_t = (g_t)_*$ for any $t \in [0,1]$. The theory describes the set $Imm(L,\EE^n)$ of
immersions of $L$ into the Euclidean space $\EE^n$ up to regular
homotopy. In general it is rather sophisticated but it simplifies nicely in the case where $L$ is parallelizable.

So let us assume that $T(L)$ is trivial and fix a global tangent frame $\tau$ of it. As usual, we denote by $\mathbb{V}_{n,k}$ the Stiefel manifold of
$k$-frames of $\RR^n$. The canonical parallelism of the Euclidean space $\EE^n$ induces a canonical trivialization of the associated $(\mathbb{V}_{n,k})$-bundle. Therefore the pair $(L,\tau)$ associates
to any smooth immersion $h: L \to \EE^n$ a smooth map $\hat h_\tau: L \to
\mathbb{V}_{n,k}$, thus defining a map
$$
\chi_{\tau} : Imm(L, \EE^n) \to C^\infty(L, \mathbb{V}_{n,k})\,
$$
which depends on the choice of $\tau$. Now the main theorem of Hirsch-Smale theory \cite{Hirsch59,Sm59} can be stated as follows:

\begin{theorem}[Hirsch-Smale's theorem]\label{HiSm} For any
parallelizable $k$-manifold $L$, the image
of $\chi_\tau$ does not depend on $\tau$. Moreover, if $n\geq k+1$
or if  $n=k$ and $L$ is open, this map induces a bijection
$$
\tilde \chi_{\tau} : [Imm(L, \EE^n)] \to [L, \mathbb{V}_{n,k}]\,,
$$
from the set of regular homotopy classes of immersions onto the set
of homotopy classes of maps from $L$ to $\mathbb V_{n,k}$.
\end{theorem}

The proof of this theorem is one of the first applications of the h-principle that one can find in the literature (see e.g. \cite{Ad93} for a comprehensive account). In order to use concretely the theory, we compute explicitly the set
$[L,\VV_{n,k}]$ in some cases to be used later.

\begin{proposition}\label{prop.calcul} Let $r$ be the number of compact components of a $k$-dimensional manifold $L$. Then
\begin{enumerate}
\item $ [L,\VV_{n,k}] =
\begin{cases}
0 & \text{if $n\geq 2k+1$},\\
H^k[L, \pi_k(\VV_{n,k})]= [\pi_k(\VV_{n,k})]^r & \text{if $n=2k$}.
\end{cases}$
\item $ [L,\VV_{n,n-k}] = H^k[L, \pi_k(\VV_{n,n-k})]= [\pi_k(\VV_{n,n-k})]^r  ~~\text{for any $n > k$}$.
\end{enumerate}
\end{proposition}

\begin{proof} First let us recall that the Stiefel manifold $\VV_{n,k}$ is
$(n-k-1)$-connected, which means that the homotopy groups
$\pi_p(\VV_{n,k})=0$ for $p \in\{1,\ldots,n-k-1\}$. Now let $g^{(j)}$ be the restriction of a continuous map $g: L \to \VV_{n,k}$  to the $j$-skeleton $L^{(j)}$ of a fixed CW-decomposition of $L$. As $\VV_{n,k}$ is connected, $g^{(0)}$ is homotopic to a constant and using the homotopy extension theorem we can assume that it is in fact constant. Next proceeding by induction and using the fact that $\VV_{n,k}$ is $j$-connected, we can assume that $g^{(j)}$ is constant for any $j \leq n-k-1$. Since the condition $n\geq 2k+1$ is equivalent to $k \leq n-k-1$, we conclude from the previous observation that in this case $[L,\mathbb {V}_{n,k}]=0$.

Finally, consider the case $n=2k$. As then $\VV_{n,k}$ is $(k-1)$-connected, we can apply results in \cite{Sp66}: since $H^q(L,\pi_q(\VV_{n,k})) = 0$ and $H^{q+1}(L, \pi_q(\VV_{n,k})) = 0$ for all $q > k$, it follows from \cite[Theorem 3, p. 447]{Sp66}  that $[L,\VV_{n,k}] = H^k[L, \pi_k(\VV_{n,k})]$, which is claim 1. To prove claim 2, we proceed in a similar way after noticing that $\VV_{n,n-k}$ is $(k-1)$-connected.
\end{proof}

\begin{remark}\label{R:computV}
Many of the homotopy groups of Stiefel manifolds have been computed \cite{Pa56}, in particular we have that

$$
\pi_k(\VV_{n,n-k}) = \begin{cases}
\ZZ & \text{if $k$ is even or if $k$ is odd and $n=k+1$\,,}\\
\ZZ_2  & \text{if $k$ is odd and $n >k+1$\,.}
\end{cases}
$$
\end{remark}

The following corollary is a straightforward consequence of Hirsch-Smale Theorem \ref{HiSm} and Proposition \ref{prop.calcul} in the case where $L$ is parallelizable. The assumption of parallelizability can be removed using the whole machinery of Hirsch-Smale theory; the final result was originally due to Whitney \cite{Wh35} and Kervaire~\cite{Ke59_2}.

\begin{corollary}\label{imm}
All the immersions of a $k$-dimensional manifold $L$ in $\EE^n$ are regularly homotopic for $n\geq 2k+1$.
\end{corollary}

It is not true that the same property holds for $n\leq 2k$ (see \cite{Hu85,Ta00,SaTa02,Ju05}). Nevertheless, we have the following remarkable extension of Corollary~\ref{imm} which will prove very useful for the study of WI embeddings (see e.g. \cite{MM67}).

\begin{theorem}\label{T:reghom}
For $k\geq 2$, all embeddings of a $k$-dimensional manifold $L$ in $\EE^{2k}$ are regularly homotopic.
\end{theorem}

In the following section we shall relate Hirsch-Smale theory to weak integrability; one bad point of this theory is that since $\tilde \chi_\tau$ depends effectively on $\tau$, the ``trivial class'' of $[Imm(L, \EE^n)]$ is not intrinsically defined and it is very difficult to detect the regular homotopy classes of
immersions which contain embeddings.

An additional problem (specially concerning strong integrability) arises from
the fact that two embeddings which are regularly homotopic are not
necessarily isotopic, as is shown by the typical example of links in
$\EE^3$: they are all regularly homotopic because $[\mathbb{S}^1,
\VV_{3,1}] = 0$ but they are of course not isotopic in general.

Regarding the classification of embeddings by isotopy, a well known result is Wu's theorem (see e.g. \cite{Da72}).

\begin{theorem}[Wu]\label{T:Wu}
Let $L$ be a $k$-dimensional manifold with $k\geq 2$. If $L$ is connected, all its embeddings in $\EE^{n}$ are isotopic for $n\geq 2k+1$, in general the same result holds for $n \geq 2k+2$.
\end{theorem}

\section{Tangential characteristic class: WI embeddings in dimension $n\geq 2k+1$}\label{S:tangent}

Recall that the triviality of the tangent bundle $T(L)$ is a necessary condition for the weak integrability of an embedding $h:L\to \EE^n$. Thus using this condition, we define here a \textit{tangential characteristic class} whose vanishing will be equivalent to the fact that $h$ is WI. This will allow us to conclude that any embedding of a parallelizable $k$-manifold in $\EE^n$ is WI for $n\geq 2k+1$.

As it will become clear later, it will be more convenient to work with immersions rather than only with embeddings. To begin, we restrict to parallelizable manifolds.

\begin{const}\label{const.tangent}

Let $h:L \to \EE^n$ be an immersion of the $k$-dimensional parallelizable manifold $L$ into $\EE^n$, $n\geq k+2$.

(1) Assuming that $T(L)$ is trivial, let  $\hat \theta: T(L) \to \RR^k$ be one of its  trivializations. The pullback of the canonical frame of $\RR^k$ defines a field of frames of $T(L)$ whose image by $dh$ is a field of tangent frames to $h(L)$. Composing with the canonical parallelism of $\EE^n$, we associate to this field a map $\theta: L \to \VV_{n,k}$ from which we can recover $\hat \theta$ without ambiguity. Therefore we shall also call $\theta$ a \textit{trivialization of $T(L)$} in the sequel.

(2) In a similar way, we associate to the immersion $h$ a map $\epsilon:L \to \mathbb{G}_{n,k}$  from $L$ to the Grassmannian of $k$-planes in $\RR^n$. All these data fit together in a commutative diagram:
$$
\xymatrix{ O(k) \ar[r]^j &\mathbb{V}_{n,k} \ar[r]^p &  \mathbb{G}_{n,k} \\
     & L \ar[u]^{\theta} \ar[ru]^{\epsilon} &  \\  }\,
$$
where $p$ is the map which associates to a frame in $\VV_{n,k}$ the $k$-plane it generates in $\RR^n$, and any trivialization $\theta: L \to \mathbb{V}_{n,k}$ of  $T(L)$
is just a lift of $\epsilon$ thus verifying $\epsilon = p \circ \theta$.
\end{const}

\begin{definition} The homotopy class $[\epsilon]\in [L,\mathbb{G}_{n,k}]$ of $\epsilon$ is called
the {\it tangential characteristic class} of the immersion $h$.
\end{definition}

The vanishing of the tangential characteristic class, that is $[\epsilon]=0$, does not necessarily imply that any trivialization $\theta$ of $T(L)$ is homotopic to a constant, but it will be so for $n\geq 2k$.

\begin{lemma}\label{inject1}
If $n\geq 2k$, then for any immersion $h$ of a parallelizable $k$-manifold $L$, the following three conditions are equivalent:
\begin{enumerate}
\item $[\epsilon]=0$.
\item $[\theta]=0$ for any trivialization $\theta$ of $T(L)$.
\item There exists a trivialization $\theta_0$ of $T(L)$ which is homotopic to a constant.
\end{enumerate}
\end{lemma}
\begin{proof}
It is standard that if $n\geq 2k$,
the inclusion map $j:O(k)\to \VV_{n,k}$ is null-homotopic in the above fibration \cite[Lemma 10.15]{Wh78}. This implies that the map $p_*: [L,\mathbb{V}_{n,k}] \to [L, \mathbb{G}_{n,k}]$ induced by $p$ is injective and hence $[\epsilon]=0$ if and only if $[\theta]=0$ for any lift $\theta$ of $\epsilon$, thus proving that conditions 1 and 2 are equivalent. As the implications $2\Rightarrow 3 \Rightarrow 1$ are obvious, the proof is complete.
\end{proof}

Using our tangential characteristic class we obtain the following homotopical characterization of weak integrability.

\begin{proposition}\label{prop.tangent} Let $L\subset \EE^n$ be a $k$-dimensional embedded submanifold of $\EE^n$. If $m = n-k \geq 2$, the following conditions are equivalent:
\begin{enumerate}
\item $T(L)$ extends over $\EE^n$.
\item $L$ is parallelizable and $T(L)$ admits a trivialization $\theta$ which is homotopic to a constant, i.e. $[\theta]=0$.
\item $L$ is parallelizable and its tangential characteristic class is zero, i.e. $[\epsilon]=0$.
\item $L$ is WI.
\end{enumerate}
\end{proposition}

\begin{proof}

$1\Rightarrow 2$: If $T(L)$ extends over $\EE^n$, its extension $\tilde{T}(L)$ is a trivial bundle because $\EE^n$ is contractible. Therefore any trivialization $\tilde{\theta}$ of $\tilde T(L)$ induces a trivialization $\theta$ of $T(L)$ which is homotopic to a constant by definition.

$2\Rightarrow 3$: This is because $[\epsilon] = p_*[\theta]$.

$3\Rightarrow 4$: If $[\epsilon]=0$, we use the homotopy extension theorem to extend $\epsilon$ as a map $\tilde \epsilon: \EE^n \to \GG_{n,k}$. It defines a bundle $\tilde T(L)$ over $\EE^n$ which is trivial and extends $T(L)$. The result now follows from Corollary \ref{cor.main}.

$4\Rightarrow 1$: Again by Corollary \ref{cor.main}.
\end{proof}

\begin{corollary}\label{cor.components}
Let $L$ be an embedded $k$-submanifold of $\EE^n$, $n \geq k+2$. Then $L$ is WI if and only if all its connected components are WI.
\end{corollary}
\begin{proof}
By Proposition \ref{prop.tangent}, $L$ is WI if and only if it is parallelizable and its tangential characteristic class $[\epsilon] = 0$. Now $L$ is parallelizable exactly if all its connected components $L_i$ are parallelizable and the result follows from the fact that $[\epsilon] = 0$ if and only if the tangential class $[\epsilon_i]$ of each $L_i$ vanishes.
\end{proof}

\begin{remark}
Reinterpreting Corollary \ref{thm.cod1}, we see that Corollary \ref{cor.components} also holds in the case $n=k+1$.
\end{remark}

Note that by Hirsch-Smale classification, the tangential characteristic classes $[\epsilon_h]$ and $[\epsilon_{h'}]$ of any two regularly homotopic immersions $h,h': L \to \EE^n$ of a parallelizable $k$-manifold $L$ are equal. Therefore, we can formulate our first important result, which tackles the question of WI integrability in large codimension. We do not assume a priori that $L$ is parallelizable.

\begin{theorem}\label{t.aob} Let $L$ be a $k$-dimensional manifold. For $n\geq 2k+1$, the following conditions are equivalent:
 \begin{enumerate}
\item Any embedding of $L$ in $\EE^n$ is WI and thus has trivial normal bundle.
\item There exists an embedding of $L$ in $\EE^n$ which is WI.
\item $L$ is parallelizable.
\end{enumerate}
Moreover if $L$ is parallelizable and $n=2k$, then $\nu(L)$ is trivial and either all embeddings of $L$ in $\EE^n$ are WI or none is WI.
\end{theorem}
\begin{proof}
Implications $1 \Rightarrow 2 \Rightarrow 3$ are immediate by Lemma~\ref{L:easy}. Next observe that Proposition \ref{prop.calcul} implies that $[L,\VV_{n,k}]=0$ for $n\geq 2k+1$, so if $L$ is parallelizable, any trivialization of $T(L)$ is homotopic to a constant, which means that $3 \Rightarrow 1$ on account of Proposition~\ref{prop.tangent}. The three conditions are equivalent.

Now consider the case $n=2k$. Assuming that $L$ is parallelizable, it follows from Lemma~\ref{L:trivnu} that its normal bundle $\nu(L)$ is trivial and if $k\geq 2$, any two embeddings of $L$ in $\EE^n$ are regularly homotopic on account of Theorem \ref{T:reghom}. Their tangential characteristic classes are equal and the last claim follows by Proposition \ref{prop.tangent}. It extends to the case $k=1$ because any embedding of any $L$ in $\EE^2$ is WI if and only if $L$ is open, cf. Corollary \ref{thm.cod1}.

\end{proof}

\begin{corollary}\label{cor:37}
Any embedding of a $3$-manifold in $\EE^n$ is WI for $n\geq 7$; any embedding with trivial normal bundle of a $7$-manifold in $\EE^n$ is WI for $n\geq 15$.
\end{corollary}
\begin{proof}
Recalling that any $3$-manifold is parallelizable \cite{St35} and that any $7$-manifold which embeds with trivial normal bundle into $\EE^n$ for some $n\geq 8$ is also parallelizable \cite{Su64}, the result follows by Theorem \ref{t.aob}.
\end{proof}

Let us observe a striking consequence of Theorem \ref{t.aob}: for $n\geq 2k$, the weak integrability of the submanifold $L$ does not depend on the embedding, it is an intrinsic character of $L$. In fact, we shall see that there exist parallelizable $k$-submanifolds of $\EE^{2k}$ which are not WI although their normal bundle is trivial. Note also that Theorem \ref{t.aob} implies the second part of Lemma~\ref{L:trivnu} (except for the case $n=2k$) and also complements
Kervaire's theorem \cite{Ke59} which asserts that any embedding of a $k$-sphere in $\EE^n$ has trivial normal bundle if $n-k >\frac{n+1}{3}$.

\begin{appl}\label{appl.un}
The embedding of any $1$-dimensional manifold $L$ in $\EE^n$, $n\geq 3$, is WI. This result was already obtained by a different method in Application~\ref{appl.low}.
\end{appl}

\section{Normal characteristic class: WI embeddings in dimension $k+2 \leq n \leq 2k$} \label{Euc3}

In Section \ref{S:tangent} we have characterized WI embeddings of codimension $m \geq k+1$ using a homotopical obstruction related to the parallelizability of $L$. In this section we shall study WI embeddings of codimension $m \leq k$. To do so we introduce a second obstruction which is related to the triviality of the normal bundle of the embedding, another necessary condition for weak integrability.

\subsection{Normal characteristic class and weak integrability}\label{S:normal}

Before defining our ``normal characteristic class'', we first recall a precise description of the normal bundle of an immersion. Here submanifolds are no longer assumed to be parallelizable.

\begin{const}\label{const.normal}(\textit{Normal bundle of immersions - trivializations}). Let $h:L \to \EE^n$ be an immersion of the $k$-dimensional manifold $L$ into $\EE^n$, $n\geq k+1$.

(1) The pullback $h^*[T(\EE^n)]$ of the tangent bundle to $\EE^n$ contains naturally the tangent bundle $T(L)$ and comes endowed with a metric for which $\dd h$ is a fiberwise isometry. The orthogonal bundle to $T(L)$ in $h^*[T(\EE^n)]$ is the \textit{normal bundle $\nu(h)$} of the immersion $h$, the differential $\dd h$ mapping its fiber $\nu_x(L)$, $x\in L$, onto the orthogonal space to $\dd h[T_x(L)]$ in $T_{h(x)}(\EE^n)$. If $h$ is an embedding, $\nu(h)$ identifies naturally with the usual normal bundle of the submanifold $h(L)$.

(2) Next assume that $\nu(h)$ is trivial and let $\hat \sigma: \nu(h) \to \RR^{n-k}$ be one of its  trivializations. The pullback of the canonical frame of $\RR^{n-k}$ defines a field of frames of $\nu(h)$ whose image by $\dd h$ is a field of normal frames to $h(L)$ in $\EE^n$. Composing with the canonical parallelism of $\EE^n$, we associate to this field a map $\sigma: L \to \VV_{n,n-k}$ from which we can recover $\hat \sigma$ without ambiguity. We call $\sigma$ a \textit{trivialization of $\nu(h)$}.

(3) In a similar way, the immersion $h$ defines a field of normal planes to $h(L)$ by $\dd h[\nu_x(L)]$ and using again the canonical parallelism of $\EE^n$, we define a map $\eta:L \to \mathbb{G}_{n,n-k}$ from $L$ to the Grassmannian of $(n-k)$-planes in $\RR^n$. All these maps fit together in a commutative diagram:
$$
\xymatrix{ O(n-k) \ar[r]^j &\mathbb{V}_{n,n-k} \ar[r]^p &  \mathbb{G}_{n,n-k} \\
     & L \ar[u]^{\sigma} \ar[ru]^{\eta} &  \\  }\,,
$$
where any trivialization $\sigma$ of $\nu(L)$ appears as a lift of $\eta$, thus verifying $\eta = p \circ \sigma$.
\end{const}

\begin{definition} The homotopy class $[\eta] \in [L,\GG_{n,n-k}]$ of $\eta$
is called the {\it normal characteristic class}
of the immersion $h$.
\end{definition}

As for the tangent class, the fact that the normal characteristic class of an immersion vanishes, i.e. $[\eta] = 0$, does not generally imply that the homotopy class of any trivialization of $\nu(h)$ is trivial. Nevertheless, this property holds for $n\leq 2k$, a remarkable fact which will be systematically exploited in the next subsections.

The proofs of the following results go along the same lines as those of Lemma~\ref{inject1} and Proposition \ref{prop.tangent} respectively but using the normal characteristic class instead of the tangent one, so we will omit these proofs.

\begin{lemma}\label{inject}
For $k+1\leq n\leq 2k$ and any immersion $h:L \to \EE^n$ of a $k$-dimensional manifold $L$, the following three conditions are equivalent:
\begin{enumerate}
\item The normal characteristic class of $h$ vanishes, i.e. $[\eta]=0$.
\item $[\sigma]=0$ for any trivialization $\sigma$ of $\nu(h)$.
\item There exists a trivialization $\sigma_0$ of $\nu(L)$ which is homotopic to a constant.
\end{enumerate}
\end{lemma}

In particular, we get a new homotopic characterization of weak integrability in terms of the normal class.

\begin{proposition}\label{prop.normal} Let $L\subset \EE^n$ be a $k$-dimensional submanifold of $\EE^n$. If the codimension $m = n-k \geq 2$, the following conditions are equivalent:
\begin{enumerate}
\item $\nu(L)$ extends over $\EE^n$.
\item $\nu(L)$ is trivial and admits a trivialization $\sigma$ which is homotopic to a constant, i.e. $[\sigma]=0$.
\item The normal characteristic class of $L$ is zero, i.e. $[\eta]=0$.
\item $L$ is WI and thus parallelizable.
\end{enumerate}
\end{proposition}

We could also deduce from Proposition \ref{prop.normal} another proof of Corollary~\ref{cor.components} using the normal characteristic class instead of the tangent one. The next corollary is obvious from Propositions \ref{prop.tangent} and \ref{prop.normal}.

\begin{corollary}For any immersion $h:L \to \EE^n$ with trivial normal bundle of a parallelizable manifold $L$, $[\eta]=0$ if and only if $[\epsilon]=0$.
\end{corollary}

Next we want to recall that the normal class of an immersion is invariant by regular homotopy. The proof is standard, but we sketch it for the sake of completeness. In turn, it also provides an alternative proof of the claim in Theorem~\ref{t.aob} concerning embeddings into $\EE^{2k}$ using the normal map.

\begin{proposition}\label{lem.equivalence}Let $h$ and $h'$  be two regularly homotopic immersions of a $k$-manifold $L$ in $\EE^n$. Then the normal bundles $\nu(h)$ and $\nu(h')$ are equivalent and the corresponding normal characteristic classes $[\eta_h]$ and $[\eta_{h'}]$ are equal. Moreover, if $h$ and $h'$ are embeddings, they are simultaneously WI or not. In particular, all embeddings of $L$ in $\EE^{2k}$ are WI or none is WI.
\end{proposition}

\begin{proof} First assume that the regular homotopy is defined by a smooth map $g: L \times [0,1] \to M$. Then the $1$-parameter family of normal bundles $\nu(g_t)$ of the immersions $g_t$ fit together defining a bundle over $L \times [0,1]$. Consequently, all these bundles are equivalent and the spaces of homotopy classes $[L \times  \{t\}, \VV_{n,n-k}]$ identify canonically. In particular, $\nu(h)$ is equivalent to $\nu(h')$ and $[\eta_h]=[\eta_{h'}]$.
The general case is analogous as the bundles $\nu(g_t)$ again define a global bundle due to the continuity assumption in the definition of regular homotopies.

Finally, as in the proof of Theorem~\ref{t.aob}, the fact that any two embeddings $h$ and $h'$ of $L$ in $\EE^{2k}$ are simultaneously WI or not is a consequence of Theorem~\ref{T:reghom}.
\end{proof}

The normal characteristic class and Proposition \ref{prop.normal} can be used to characterize WI open submanifolds. This is the main result of this subsection.

\begin{theorem}\label{thm.open} Let $L \subset \EE^n$ be an embedded $k$-dimensional open submanifold, $k\leq n-1$. Then $L$ is WI if and only if $\nu(L)$ is trivial. In particular, the fact that the normal bundle $\nu(L)$ of $L$ is trivial implies that $L$ is parallelizable.
\end{theorem}

\begin{proof}
The case $k=n-1$ is a consequence of Corollary \ref{thm.cod1}, so let us assume that $k\leq n-2$. As $L$ is open, it follows that $[L,\VV_{n,n-k}] = 0$ by Proposition~\ref{prop.calcul}. This shows that any trivialization $\sigma$ of the normal bundle $\nu(L)$ is in the trivial homotopy
class and we conclude by means of Proposition \ref{prop.normal}.
\end{proof}

As any codimension $2$ submanifold of $\EE^n$ has trivial normal bundle, cf. Lemma \ref{L:trivnu}, we conclude from Theorem \ref{thm.open} that any proper embedding of an open $k$-manifold $L$ in $\EE^{k+2}$ is WI.

\begin{appl} It is a standard fact~\cite{Wh61} that any open orientable manifold $L$ of dimension $k\leq 3$ is parallelizable. We recover this result just realizing that $L$ can be
properly embedded into $\EE^{2k-1}$ \cite{Hi61}. As the codimension of
this embedding is $\leq 2$, its normal bundle is trivial and the result follows by Theorem~\ref{thm.open}.
\end{appl}

\begin{appl}
Let $L$ be a non-singular algebraic hypersurface of $\CC^n$. By definition $L$ is given as the zero set of a complex polynomial, i.e.
$$L=\{z\in\CC^n:P(z)=0\}\,,$$
with $\nabla P|_{L}\neq 0$. Identifying $\CC^n$ with $\EE^{2n}$ we have that $L$ is a codimension $2$ real-submanifold, which has trivial normal bundle and is necessarily open. Theorem~\ref{thm.open} implies that the tangent bundle of $L$ is differentially trivial.
\end{appl}

\subsection{Total weak integrability}\label{S:TWI}

Here we introduce a family of manifolds for which weak integrability is a completely intrinsic character of the manifold, not at all depending on the embeddings.

\begin{definition}
A parallelizable $k$-manifold $L$ will be called {\it totally weakly integrable (TWI)} if for any $n \geq k+2$, any embedding of $L$ in $\EE^n$ is WI if and only if it has trivial normal bundle.
\end{definition}

Reinterpreting previous results, we make the following observations:

\medskip

\noindent (1) Any open manifold is TWI by Theorem \ref{thm.open}.

\noindent (2) To study total weak integrability of non-open manifolds, it will suffice to consider compact connected $k$-manifolds on account of Corollary~\ref{cor.components}.

\noindent (3) In the case of compact connected parallelizable $k$-manifolds, we already know by Theorem \ref{t.aob}, that any embedding into $\EE^n$ is WI for $n \geq 2k+1$. Therefore such a manifold will be TWI exactly if the corresponding property also holds for $k+2 \leq n \leq 2k$ for embeddings with trivial normal bundle.

\medskip

Therefore the goal of this subsection is to characterize the compact connected manifolds which are TWI. Let us first make a preliminary construction which will be frequently used in the sequel.

\begin{const}\label{const.fattening}(\textit{Fattening immersions and embeddings}).

(1) Observe that the natural embedding $\rho:\EE^n \to \EE^p = \EE^n \times \EE^{p-n}$ induces an embedding $r: \VV_{n,n-k}\to \VV_{p,p-k}$ which is defined by adding a $(p-n)$-normal frame to a tangent frame of $\EE^n \subset \EE^p$. For any immersion $h: L \to \EE^{n}$, the composition $g= \rho \circ h: L \to \EE^{p}$ is an immersion of $L$ into $\EE^{p}$ which will be called the \textit{$p$-fattening} of $h$. It satisfies the following properties:

\noindent (a) The normal bundle of $g$ is $\nu(g) = \nu(h) \oplus \mu$, where $\mu$ is the tangent bundle of $\EE^{p-n}$ in $\EE^{p}$. As $\mu$ is trivial, $\nu(g)$ will be trivial as soon as $\nu(h)$ is trivial. Furthermore, adding a $(p-n)$-frame to a trivialization $\sigma_h:L\to \VV_{n,n-k}$ of the normal bundle $\nu(h)$, we obtain a trivialization $\sigma_g$ of $\nu(g)$ which satisfies $\sigma_g = r \circ \sigma_h$.

\noindent (b) Similarly, denoting by $s: \GG_{n,n-k}\to \GG_{p,p-k}$ the corresponding
embedding of the Grassmanians, we see that $\eta_g = s \circ \eta_h$.

(2) To this fattening procedure, we associate a fibration of Stiefel manifolds

$$
\xymatrix{
\VV_{n,n-k}\ar[r]^r & \VV_{p,p-k} \ar[r] & \VV_{p,p-n}
}\,,
$$
whose associated homotopy exact sequence reads as:

$$
\xymatrix{\cdots \to \pi_{j+1}(\VV_{p,p-n}) \to \pi_{j}(\VV_{n,n-k}) \ar[r]^{r_*} & \pi_{j}(\VV_{p,p-k}) \to \pi_{j}(\VV_{p,p-n}) \to \cdots}
$$
\end{const}

The following remarkable dichotomy complements the corresponding result for $n=2k$ (cf. Theorem \ref{t.aob}) and shows that the WI character of the embeddings mostly depends on the manifold itself.

\begin{proposition}\label{prop.small}
For any $k$-dimensional manifold $L$, one of the following conditions holds:
\begin{enumerate}
\item For any $k+2\leq n\leq 2k$, any embedding of $L$ into $\EE^n$ with trivial normal bundle is WI and in particular $L$ is parallelizable.
\item No embedding of $L$ into $\EE^n$ is WI for $k+2\leq n\leq 2k$.
\end{enumerate}
\end{proposition}

\begin{proof} We proceed in two steps using the definitions and notations introduced in Construction \ref{const.fattening}.

(1) First consider an embedding $h:L \to \EE^n$, $n \leq 2k$, with trivial normal bundle and its $2k$-fattening $g$.
Since $\VV_{2k,2k-n}$ is $(n-1)$-connected, our assumption $n \geq k+2$ implies that $\pi_{k+1}(\VV_{2k,2k-n}) = \pi_{k}(\VV_{2k,2k-n}) = 0$,
which in turn shows that the induced homomorphism  $r_*:\pi_{k}(\VV_{n,n-k})\to  \pi_{k}(\VV_{2k,k})$ in the corresponding exact sequence is an isomorphism. Accordingly, the two classes $[\sigma_h]$ and $[\sigma_g] = r_*([\sigma_h])$ are simultaneously zero or not. As $n\leq 2k$, it then follows from Lemma \ref{inject} that the normal characteristic classes $[\eta_h]$ and $[\eta_g]$ vanish simultaneously or not, and hence the two embeddings $h$ and $g$ are simultaneously WI or not on account of Proposition \ref{prop.normal}.

(2) Now assume that $h$ is WI and take another embedding  $h': L \to \EE^{n'}$ with trivial normal bundle and $k+2 \leq n' \leq 2k$. The $2k$-fattening $g'$ of $h'$ is regularly homotopic to $g$, cf. Theorem \ref{T:reghom}, and hence $g'$ is WI because $g$ is so by Proposition \ref{lem.equivalence}. The discussion in the previous item finally implies that $h'$ is WI and the proposition follows.
\end{proof}

Combining Proposition \ref{prop.small} with Theorem \ref{t.aob}, we get the main result of this subsection, whose proof is straightforward.

\begin{theorem}\label{T:TWI1}For any $k$-dimensional parallelizable manifold $L$, the following conditions are equivalent:
\begin{enumerate}
\item $L$ is TWI.
\item There exists $k+2 \leq n \leq 2k$ and at least one embedding $h: L \to \EE^n$ which is WI.
\end{enumerate}
\end{theorem}

Let us finish by describing some applications of Theorem~\ref{T:TWI1}.

\begin{proposition}\label{P:tori}Any parallelizable $k$-manifold $L$ whose compact components are tori $\TT^k$, $k\geq 1$, is TWI.
\end{proposition}

\begin{proof} On account of Corollary \ref{cor.components}, we can restrict to the case of an individual torus $\TT^k$. It is then enough to show that there exists an embedding of $\TT^k$ into $\EE^{k+2}$ which is WI, cf. Theorem \ref{T:TWI1}. To prove this claim we proceed in three steps. The case $k=1$ being already covered by Application \ref{appl.un}, we assume that $k\geq 2$.

(1) Considering the solid torus $K^3 := \mathbb{S}^1 \times D^2$ as a $3$-dimensional Reeb component, we see that it admits a non-vanishing vector field $X_3$ transverse to the boundary $\partial K^3 = \TT^2$. We construct a similar vector field $X_k$ on $K^k:=\TT^{k-2}\times D^2= \TT^{k-3} \times K^3$ by lifting $X_3$ transversally to the fibers $\TT^{k-3}$.

(2) Next we claim that for any $k \geq 2$, the torus $\TT^k$ embeds into $\EE^{k+1}$ in such a way that it bounds a manifold diffeomorphic to $K^{k+1}$. Indeed, this is the case for the standard embedding of $\TT^2$ into $\EE^3$ and we go on by induction on $k$. Having defined $\TT^{k-1} \subset \EE^k \subset \EE^{k+1}$, we recall that its normal bundle in $\EE^{k+1}$ is trivial by Lemma \ref{L:trivnu}, and thus any tubular neighborhood $N_{k+1}$ of $\TT^{k-1}$ is diffeomorphic to $K^{k+1}$. Moreover $\partial N_{k+1}$ is an embedding of $\TT^k$ into $\EE^{k+1}$ which bounds a manifold diffeomorphic to $K^{k+1}$.

(3) Now proceeding as in Application \ref{appl.low}, we extend the vectorfield $X_{k+1}$ on $N_{k+1}$ as a non-vanishing vector field $Y_{k+1}$ on $\EE^{k+1}$ which is transverse to $\TT^k$. Finally, we consider the $(k+2)$-fattening $\TT^k \subset \EE^{k+1} \subset \EE^{k+2}$ of the previous embedding of $\TT^k$ in $\EE^{k+1}$. Due to the existence of $Y_{k+1}$, its normal bundle extends all over $\EE^{k+2}$ and so it is WI on account of Proposition \ref{prop.normal}. Consequently $\TT^k$ is TWI by Theorem \ref{T:TWI1}.
\end{proof}

To find examples of parallelizable manifolds which are not TWI, we now focus on spheres. Recall that the only parallelizable spheres are $\e^1$, $\e^3$ and $\e^7$ (see \cite{BM58}). We already know from Application \ref{appl.un} that $\e^1$ is TWI, and from Theorem \ref{t.aob} that all the embeddings of $\e^k$ in $\EE^n$ are WI for $k\in\{3,7\}$ and $n\geq 2k+1$. The following result covers Theorem D of Section~\ref{S.mostrelevant}.

\begin{proposition}\label{spheres} In both cases $k=3,7$, no embedding of $\e^k$ in $\EE^n$ is WI for any $k+2\leq n\leq 2k$. Therefore $\e^k$ is not TWI.
\end{proposition}

\begin{proof}The proof is similar to that of Proposition \ref{prop.small}. Using the notations of Construction \ref{const.fattening}, we consider the following fibration of Stiefel manifolds and its associated homotopy sequence:
$$
\xymatrix{
\e^k = \VV_{k+1,1}\ar[r]^r & \VV_{k+2,2} \ar[r] & \VV_{k+2,1}= \e^{k+1}}
$$
$$
\xymatrix{\cdots \to \pi_{k+1}(\e^{k+1})=\ZZ \to \pi_{k}(\e^k)=\ZZ \ar[r]^{r_*} & \pi_{k}(\VV_{k+2,2}) \to \pi_k(\e^{k+1})=0 \to \cdots}
$$

The standard embedding $h: \e^k \to \EE^{k+1}$ has trivial normal bundle and therefore so does its $(k+2)$-fattening $g: \e^k \to \EE^{k+2}$. The standard trivialization $\sigma_h:\e^3\to \VV_{4,1}=\e^3$ of $\nu(h)$ is the generator of $\pi_3(\e^3)$, so $[\sigma_h]\neq 0$. Moreover, the composition $\sigma_g = r \circ \sigma_h$ is a trivialization of $\nu(g)$ and, since $r_*$ is onto, it follows that $[\sigma_g] =r_*[\sigma_h]$ generates $\pi_{3}(\VV_{5,2})=\ZZ_2$ and so $[\sigma_g]\neq 0$.

Finally, Lemma \ref{inject} implies that the normal class of the fattening is not trivial, i.e. $[\eta_g]\neq 0$, thus showing that the embedding $g$ is not WI and the claim follows applying Proposition \ref{prop.small}.
\end{proof}

Using more sophisticated tools from the theory of immersions, we will provide in the next two subsections a complete classification of manifolds which are TWI.

\subsection{TWI manifolds of dimension $k \neq 3,7$}\label{S:main}

In this subsection we shall prove that for $k\neq 3,7$, the necessary condition for a parallelizable $k$-submanifold $L\subset \EE^n$ to be WI, i.e. triviality of the normal bundle, is also sufficient thus showing that all parallelizable manifolds of dimension $k\neq 3,7$ are TWI.

The main trick to prove this result consists in relating our embeddings to codimension $1$ immersions, for which our normal class reduces to the usual degree studied by Hopf and other authors (see \cite{Ho25} and \cite{Ho26}). We first review the corresponding results to be used in the sequel.

\begin{const}\label{const.degree}(\textit{Degree and normal class of codimension $1$ immersions}). Consider an immersion $h:L \to \EE^{k+1}$ of a compact and connected $k$-dimensional manifold $L$. At each point $y \in h(L)$ the given orientations of $L$ and $\EE^{k+1}$ determine a unique normal vector $\nu_y$ thus defining a canonical trivialization $\sigma_h:L \to \VV_{k+1,1}=\e^k$ of the normal bundle $\nu(h)$. According to Proposition \ref{prop.calcul} its homotopy class is an element
$$
[\sigma_h] \in [L,\e^k] = H^k(L,\pi_k(\e^k)) = \pi_k(\e^k) = \ZZ\,,
$$
so it identifies with the usual \textit{degree} of $\sigma_h$, which was called by Hopf the \textit{curvatura integra} of $h$; we denote it by $Ci(h)$ and describe it in some cases which are of interest to us.

(1) For the immersion $h$ of a parallelizable $k$-manifold $L$, Bredon-Kosinski~\cite{BK66} (see also Thomas~\cite{Th69}) proved that
$$
Ci(h)=[\sigma_h] = \begin{cases}
0 & \text{if $k$ is even\,,}\\
0 ~~\text{mod $2$} & \text{if $k$ is odd and $k\notin \{1,3,7\}$\,.}
\end{cases}
$$

(2) Let $h_j:L_j \to \EE^{k+1}$, $j \in \{1,2\}$, be two immersions of compact connected $k$-manifolds separated by a $k$-plane. Connecting them by means of a tube, we define their connected sum $h= h_1\sharp h_2: L = L_1 \sharp L_2 \to \EE^{k+1}$ whose curvatura integra is given by the following formula of Milnor \cite{Mi56}:
$$
Ci(h) = Ci(h_1) + Ci(h_2) - 1\,.
$$

(3) Finally we observe that by Lemma \ref{inject}, the normal characteristic class $[\eta]$ of any codimension $1$ immersion $h$ vanishes if and only if $Ci(h)=0$.
\end{const}

\begin{const}\label{const.compression}(\textit{Compression of embeddings}). According to Lemma \ref{L:trivnu}, any embedding $h:L \to \EE^{2k}$ of a compact connected and parallelizable $k$-manifold $L$ has trivial normal bundle. Now it is proved in \cite{Hirsch59} that $h$ is regularly homotopic to an immersion $h':L\to \EE^{k+1}\times \EE^{k-1}$ which is transverse to the factor $\EE^{k-1}$. The canonical projection of $\EE^{2k}$ onto $\EE^{k+1}$ defines a codimension $1$ immersion $f: L \to \EE^{k+1}$ which will be called a \textit{compression of~$h$}. Denote by $g$ the $2k$-fattening of $f$ (cf. Construction \ref{const.fattening}), and consider the 1-parameter family of maps $\lambda_t: \EE^{2k}= \EE^{k+1}\times \EE^{k-1} \to \EE^{k+1}\times \EE^{k-1}$ given by $\lambda_t(u,v):= (u, tv)$ for $t \in [0,1]$. Then $g_t:= \lambda_t \circ h'$ defines a regular homotopy between $g$ and $h'$, thus proving that the initial embedding $h$ is regularly homotopic to $g$.

If we assume that $L$ is compact, connected and has non-empty boundary, then any embedding $h:L\to \EE^{n}$ ($n\geq k+1$) with trivial normal bundle can be compressed to an (equidimensional) immersion $f:L\to \EE^k$, see e.g. the proof of Theorem~4.7 in~\cite{Hi61}. Proceeding exactly as in the previous paragraph, it follows that the $n$-fattening $g$ of $f$ is regularly homotopic to $h$.
\end{const}

The following proposition is the key result of this subsection.

\begin{proposition}\label{prop.key} Let $h:L \to \EE^{2k}$ be an embedding of a compact connected $k$-manifold $L$ and let $f: L \to \EE^{k+1}$ be a compression of $h$. Then $h$ is WI if and only if for the canonical trivialization $\sigma_f$ of the normal bundle of $f$ we have
$$
Ci(f)=[\sigma_f] = \begin{cases}
0 & \text{if $k$ is even}\,,\\
0 ~~\text{mod.~ $2$} & \text{if $k$ is odd}\,.
\end{cases}
$$
\end{proposition}

\begin{proof}

Indeed, as the $2k$-fattening $g$ of $f$ is regularly homotopic to $h$, we see, according to Proposition \ref{lem.equivalence} and Lemma \ref{inject}, that $h$ is WI if and only if the associated trivialization $\sigma_g$ of the normal bundle of $g$ is homotopic to zero, i.e. $[\sigma_g]=0$. The case $k=1$ being trivial, assume that $k\geq 2$ and consider the fibration and homotopy exact sequence associated to $g$:
$$
\xymatrix{
\e^k=\VV_{k+1,1}\ar[r]^r & \VV_{2k,k} \ar[r] & \VV_{2k,k-1}\,,
}
$$
$$\xymatrix{
\cdots \ar[r]&  \pi_k(\e^k)=\ZZ\ar[r]^{r_*}&  \pi_k(\VV_{2k,k})\ar[r] &  \pi_k(\VV_{2k,k-1})=0\ar[r]&\cdots}
$$
where $[\sigma_g] = r_*([\sigma_f])$. There are two possibilities:

(a) If $k$ is even, $\pi_k(\VV_{2k,k}) = \ZZ$ and $r_*:  \pi_k(\e^k) \to  \pi_k(\VV_{2k,k})$ is an isomorphism. Therefore $[\sigma_g]=0$ if and only if $[\sigma_f]=0$.

(b) If $k$ is odd, $\pi_k(\VV_{2k,k}) = \ZZ_2$, the homomorphism $r_*$ is onto and $[\sigma_g]=0$ if and only if $[\sigma_f]$ belongs to the kernel of $r_*$, which is equal to the subgroup $2\ZZ \subset \pi_k(\e^k)= \ZZ$.

The proof is complete.
\end{proof}

We reach to the main theorem of this subsection, which includes the first claim in Theorem C of Section~\ref{S.mostrelevant}.

\begin{theorem}\label{T:classif}
For $k \notin \{3,7\}$, any $k$-dimensional parallelizable manifold $L$ is TWI.
\end{theorem}
\begin{proof} Again we can assume that $L$ is compact and connected.
Such a manifold embeds into $\EE^{2k}$ by the celebrated theorem of Whitney \cite{Wh44} and any such embedding $h$ has trivial normal bundle on account of Lemma \ref{L:trivnu}.

The case $k=1$ is already handled in Application \ref{appl.un}, so assume that $k\geq 2$ and let $f$ be a compression of $h$. Combining the result of Bredon-Kosinski concerning the curvatura integra $Ci(f)$, cf. Construction \ref{const.degree}, with Proposition \ref{prop.key}, we see that $h$ is WI. Applying Theorem \ref{T:TWI1}, we conclude that $L$ is TWI as we desired to prove.
\end{proof}

\begin{example} Due to the definition of TWI manifolds, it becomes relevant to ask whether the normal bundle of any embedding of such a manifold $L$ is trivial. Indeed, in \cite{HS63} one can find an explicit example of a $22$-dimensional parallelizable manifold $L_0$ such that any of its embeddings in $\EE^{30}$ has non-trivial normal bundle, but for $31\leq n\leq 37$ there are both embeddings into $\EE^n$ with trivial and non-trivial normal bundle. Of course, the embeddings of $L_0$ with non-trivial $\nu(L_0)$ cannot be WI but $L_0$ has WI embeddings into $\EE^{44}$ (see Lemma~\ref{L:trivnu}). In fact, this manifold $L_0$ is TWI according to Theorem \ref{T:classif}.
\end{example}

As is apparent from the proof of Theorem \ref{T:classif}, the crucial result to obtain weak integrability is Bredon-Kosinski's theorem, which allows us to conclude, using the compression method, that the embeddings of $L$ into $\EE^{2k}$ correspond to the zero class in the bijection defined by the Hirsch-Smale map $\tilde \chi_\tau$.

\begin{example}For $k\neq 3, 7$, the following $k$-manifolds are parallelizable and therefore TWI in view of Theorem \ref{T:classif}:
\begin{enumerate}
\item Any surface whose compact components are tori; any torus $ \mathbb{T}^k$ (compare with Proposition \ref{P:tori}).
\item Any product $\e^{k_1}\times\cdots\times \e^{k_r}$ of $r\geq 2$ spheres with at least one which is odd-dimensional \cite{St67}.
\item Any product $L_1\times\cdots\times L_r$ of $r\geq 2$ manifolds of dimension $3$ (see \cite{St35}).
\item Any Lie group or homogeneous space of the form $G/T$ where $G$ is a compact connected Lie group and $T$ a non-maximal toral subgroup, e.g. \cite{Si82}.

\end{enumerate}
\end{example}

\subsection{TWI manifolds of dimension $3$ and $7$}\label{Euc4}

In this subsection we shall assume that $k\in\{3,7\}$ unless otherwise stated. Our goal is to complete the characterization of TWI manifolds and, in particular, to provide examples of parallelizable manifolds which are not TWI. Again we can restrict to the case of manifolds which are compact and connected. 

Let us recall that for $n\geq 2k+1$, the weak integrability of $k$-manifolds embedded in $\EE^n$ follows from Corollary \ref{cor:37}. On the other hand, any $3$-manifold or any parallelizable $7$-manifold can be immersed into $\EE^{k+1}$ with arbitrary curvatura integra \cite{Mi56}. This is the reason why it is more complicated to characterize WI $3$- and $7$-manifolds than WI manifolds of dimension $k\neq 3,7$.

The following proposition is key in order to prove the main theorem of this subsection, which completes the proof of Theorem~$C$ in Section~\ref{S.mostrelevant}. It is the analogue of Bredon-Kosinski's formula (cf. Construction~\ref{const.degree}) for codimension one immersions of (parallelizable) $3$- and $7$-manifolds which are compressions of embeddings. The proof of this result is essentially due to M. Takase, to whom we are very grateful for allowing us to reproduce it here. Let us recall that the semicharacteristic $\chi^*(L)$ of a compact $k$-manifold $L$ is defined as:
$$\chi^*(L)=\sum_{i=0}^{k'}\text{rank}\,H_i(L;\ZZ_2)\,,$$
where we have set $k=2k'+1$. 

\begin{proposition}\label{prop.3and7}
Let $L$ be a compact and connected parallelizable manifold of dimension $k$, $k\in\{3,7\}$. Then any embedding $h:L\to \EE^{2k}$ is regularly homotopic to an embedding $h': L\to \EE^{2k}$ whose compression to $\EE^{k+1}$ is an immersion $f$ such that $Ci(f)=\chi^*(L)$ mod. $2$.
\end{proposition}
\begin{proof}
(i) Let us start with the case $k=3$. By~\cite{Hi612} there exists an embedding $h':L\to \EE^6$ such that $h'(L)=\partial V$, where $V$ is a compact simply connected $4$-manifold embedded in $\EE^6$ with trivial normal bundle, we call $\hat h: V\to \EE^6$ this embedding. Theorem~\ref{T:reghom} ensures that $h$ and $h'$ are regularly homotopic. According to Construction~\ref{const.compression} there is an immersion $\hat f:V\to \EE^4$ which is a compression of $\hat h$. Therefore $f:=\hat f|_{\partial V}$ is a compression of $h'$, and the result follows just realizing that $Ci(f)=\chi(V)=\chi^*(L)$ mod. $2$ by~\cite[Theorem 3bis]{Ha60} and~\cite{BK66}.

(ii) Now let us assume that $k=7$. We first prove that for any parallelizable $7$-manifold $L$ there is a compact stably parallelizable $8$-manifold $V$ with boundary $\partial V=L$; moreover $V$ can be taken such that $\pi_1(V)$, $\pi_2(V)$ and $\pi_3(V)$ are trivial. Indeed, recalling that $SO$ is the infinite special orthogonal group and that $\pi_i^S$ is the $i$-th stable homotopy group of the spheres, it is standard that the $J$-homomorphism $J_r:\pi_7(SO(r))\to \pi_{r+7}(\e^r)$ in the stable range, that is for $r\geq 9$, defines a map $J:\pi_7(SO)\to \pi_7^S$ which is onto because $\text{order}\,(\text{im}\,J)=240$ by~\cite[Theorem 1.6]{Ad66} and $\pi_7^S=\ZZ_{16}\oplus\ZZ_5\oplus\ZZ_3$ has the same order. The cokernel $\pi_7^S/\text{im}\,J$ of $J$ is hence trivial and a direct application of a result of Milnor~\cite[Theorem 6.7]{Mi59} implies that $L$ bounds a compact stably parallelizable $8$-manifold. In fact, a standard surgery procedure on this manifold allows to transform it without altering its boundary into a compact stably parallelizable $8$-manifold $V$ which is $3$-connected~\cite[Section 4]{Mi61}, as we desired to prove.

The double manifold $DV$, which is a boundaryless compact $8$-manifold homeomorphic to $\partial(V\times[0,1])$, is obviously stably parallelizable. A straightforward application of Van Kampen's theorem implies that $\pi_1(DV)=\pi_1(V)=0$, so it is simply connected. We conclude from~\cite{DS65} that $DV$, and hence $V$, embeds in $\EE^{14}$ with trivial normal bundle, so there exists an embedding $h': L\to \EE^{14}$ such that $h'(L)=\partial V$, and $h'$ is regularly homotopic to $h$ by Theorem~\ref{T:reghom}. Now we proceed as in item (i), calling $\hat f: V\to \EE^8$ an immersion which is a compression of the embedding of $V$ in $\EE^{14}$ and noticing that $f:=\hat f|_{\partial V}$ is a compression of $h'$. Since $Ci(f)=\chi(V)=\chi^*(L)$ mod. $2$ by~\cite[Theorem 3bis]{Ha60} and~\cite{BK66}, the statement of the proposition follows. 
\end{proof} 

\begin{theorem}\label{thm.caract}
Let $L$ be a $k$-manifold, $k\in\{3,7\}$. Then $L$ is TWI if and only if each compact component $L_i$ is parallelizable and $\chi^*(L_i)=0$ mod. $2$.
\end{theorem}
\begin{proof}
It is enough to consider the case that $L$ is compact and connected. The claim follows by combining Theorem~\ref{T:TWI1} with Propositions~\ref{prop.key} and~\ref{prop.3and7}.
\end{proof}

\begin{appl}\label{appl.caract}

A straightforward application of Theorem~\ref{thm.caract} yields the following results:

(1) No $\ZZ_2$-homology sphere of dimension $k\in\{3,7\}$ is TWI. This holds in particular for the sphere $\e^k$, thus recovering Proposition~\ref{spheres}.

(2) For a $3$-dimensional lens space $L_{p,q}$, it is ready to compute that $\chi^*(L_{p,q})=p$ mod. $2$, therefore $L_{p,q}$ is TWI if and only if $p$ is even.

(3) Let $\Sigma$ be a compact oriented surface, then the product $L:=\Sigma \times \e^1$ is TWI.
\end{appl}

\begin{remark}
As $\pi_3(\VV_{5,3}) = \ZZ$, the Hirsch-Smale Theorem \ref{HiSm} implies that there are infinitely many regular homotopy classes of
immersions of $\e^3$ in $\EE^5$. Moreover, there are also infinitely many of these classes which contain embeddings~\cite{Hu85} and Theorem~\ref{thm.caract} shows that none of these embeddings is WI. Concerning $\e^7$, recall that the normal bundle of any embedding is trivial provided that $n\neq 11$ \cite{Ke59,Ma59}. To the best of our knowledge it is not yet known whether there exist embeddings of $S^7$ in $\EE^{11}$ with non-trivial normal bundle.
\end{remark}

It is interesting to observe that the weak integrability of a connected sum $L=L_1\sharp\cdots\sharp L_r$ of compact $k$-manifolds can be understood in terms of the integrability of its components $L_j$. For example, this allows to restrict to prime manifolds when studying the integrability of $3$-manifolds~\cite{Mi62}. For the sake of completeness, we state the following result for any odd dimensional manifold, which is not assumed  a priori to be parallelizable.

\begin{proposition}\label{P:sums}For $k\geq2$ and odd, a connected sum
$L=L_1\sharp\cdots\sharp L_r$ of compact connected $k$-manifolds which embed in $\EE^{2k}$ with trivial normal bundle (this holds e.g. if they are parallelizable), is TWI if and only if:

$$\sum_{j=1}^r Ci(f_j)= r-1\,\,\text{mod.}~2\,.$$

\end{proposition}
\begin{proof}
Let $h_j:L_j\to\EE^{2k}$ be Whitney embeddings of the manifolds $L_j$ two by two separated by $k$-planes and let $h$ be their connected sum. By assumption, all these embeddings have trivial normal bundle, and hence we can compress them to immersions $f_j$ and $f$ into $\EE^{k+1}$ (see Construction \ref{const.compression}). Furthermore, it is easy to check that $f$ is the connected sum of all the $f_j$ and hence we can apply claim (2) of Construction \ref{const.degree} to conclude that
$$
Ci(f) = \sum_{j=1}^r Ci(f_j) - (r-1)\,.
$$
Proposition \ref{prop.key} asserts that $h$ is WI (and hence $L$ is parallelizable) if and only if $Ci(f) = 0$ \text{mod}. $2$. We conclude using Theorem \ref{T:TWI1}.
\end{proof}

\begin{corollary}\label{C:sums}
If $L=L_1\sharp\cdots\sharp L_r$ is the connected sum of $r$ compact and connected odd-dimensional $k$-manifolds, then
\begin{enumerate}
\item If all $L_j$ are TWI, then $L$ is TWI if and only if $r$ is odd. For example, the connected sum of an even number of copies of $\TT^3$ is not TWI.
\item If no $L_j$ is TWI, neither is $L$.

\item If $k\notin \{3,7\}$ and each $L_j$ is parallelizable, $L$ is TWI (and hence parallelizable) if and only if $r$ is odd.
\end{enumerate}
\end{corollary}

\section{Strong integrability of $k$-manifolds in $\EE^n$ for $n\neq 2k+1$}\label{Euc5}

The aim of this section is to clarify the relationship between SI and complete intersection in $\EE^n$ for $n\neq 2k+1$, which will allow us to improve the results of Section~\ref{s.cod2} and to tackle Question~\ref{conjci} in most cases. Let us first introduce a preliminary technical lemma. 

\begin{lemma}\label{lem.lift} Let $L$ be a $k$-dimensional manifold. If there is a WI embedding of $L$ in $\EE^n$, there also exists a SI embedding of $L$ in $\EE^{p}$ for any $p \geq n+1$.
\end{lemma}

\begin{proof} It is clearly sufficient to prove the lemma for $p=n+1$. So given a weak equation $\Phi: \EE^n \to \EE^{n-k}$ of $L$ in $\EE^n$, fix a closed tubular neighborhood $N_1$ of $L$ in $\EE^n$ such that $\Phi^{-1}(0) \cap N_1 = L$. Next choose a second closed tubular neighborhood $N_0 \subset \text{int}(N_1)$ and a smooth bump function $\lambda: \EE^n \to [0,1]$ equal to $0$ on $N_0$ and to $1$ on the complement of $N_1$. It is easy to check that the map $\varphi: \EE^{n+1} \to \EE$ defined by
$$
\varphi(x,t):=[1-\lambda(x)]t + \lambda(x)e^t
$$
for $(x,t) \in \EE^n\times \EE=\EE^{n+1}$, is a submersion such that $\varphi^{-1}(0) \subset q^{-1}(N_1)$, where $q$ is the canonical projection $q:\EE^{n+1} \to \EE^n$. It follows that $\varphi^{-1}(0) \cap q^{-1}(L) = L$.

Finally, the map $\Phi':\EE^{n+1} \to \EE^{n-k+1}$ defined by
$$
\Phi'(x,t):= [\Phi(x), \varphi(x,t)]
$$
for $(x,t) \in \EE^n\times\EE=\EE^{n+1}$ is a submersion and a weak equation for $L \subset \EE^{n+1}$. Due to the choice of $N_1$, it verifies
 $$
\Phi'^{-1}(0) = \Phi^{-1}(0) \cap \varphi^{-1}(0) \subset \Phi^{-1}(0) \cap q^{-1}(N_1) \subset q^{-1}(L)
$$
and we conclude that
$$
\Phi'^{-1}(0) \subset \varphi^{-1}(0) \cap q^{-1}(L) = L\,,
$$
thus showing that $\Phi'$ is in fact a strong equation of $L$ in $\EE^{n+1}$.
\end{proof}

The next theorem extends Theorem \ref{t.aob} to cover the study of strong integrability of parallelizable manifolds in large codimension.

\begin{theorem}\label{thm.SIgreat} Any embedding of a parallelizable $k$-dimensional manifold $L$ in $\EE^n$ is SI and hence a tame complete intersection for $n\geq 2k+2$. This result also holds for $n=2k+1$ in the following restrictive cases:
\begin{enumerate}
\item $L$ is connected, TWI and $k\geq 2$.
\item $L$ is open.

\end{enumerate}
\end{theorem}

\begin{proof} Any embedding of $L$ in $\EE^{2k+1}$ is WI by Theorem \ref{t.aob} and hence, applying Lemma \ref{lem.lift}, we obtain a SI embedding of $L$ in $\EE^n$, $n\geq 2k+2$. Now observe that without any additional assumption, all embeddings of a $k$-manifold in $\EE^n$ are isotopic for $n\geq 2k+2$ on account of Wu's Theorem~\ref{T:Wu}, thus concluding that any embedding of $L$ in $\EE^n$ is SI in this case.

Now we consider the two particular cases. If $k\geq 2$ and $L$ is TWI, its Whitney embedding~\cite{Wh44} into $\EE^{2k}$ is WI and therefore it admits a SI embedding in $\EE^n$ for any $n\geq 2k+1$ according to Lemma \ref{lem.lift}. Assuming that $L$ is connected, Wu's theorem again implies that any embedding is isotopic to the former one and hence SI in $\EE^n$, $n\geq 2k+1$.

Finally, an open manifold $L \subset \EE^{2k+1}$ is CI according to~\cite{BK96} and as noticed in the proof of Theorem~\ref{thm.open}, any trivialization of $\nu(L)$ is in the trivial class, thus showing that $L$ is TCI (Definition~\ref{D:ci}). We conclude applying Theorem~\ref{thm.main2}.
\end{proof}

\begin{remark}
Any compact parallelizable $k$-manifold, $k\in\{1,3,7\}$, is WI in $\EE^{2k+1}$ by Theorem~\ref{t.aob} but we do not know whether it is SI unless it is TWI. We shall deal with this problem in the next section.
\end{remark}

Theorem \ref{thm.SIgreat} concerns $k$-embeddings in dimension $n\geq 2k+1$, so let us now focus on the case of low codimension, i.e. $k+2 \leq n \leq 2k$. The following result generalizes Theorem \ref{T:SIandTCI} and tackles Question~\ref{conjci} in most cases.

\begin{theorem}\label{T:lowdim} Let $L$ be a $k$-manifold embedded in $\EE^n$, $2+k\leq n\leq 2k$. The following conditions are equivalent:
\begin{enumerate}
\item $L$ is SI and hence a tame complete intersection.
\item $L$ is both WI and CI.
\end{enumerate}
Therefore, Question~\ref{conjci} has an affirmative answer in all Euclidean spaces for $n \neq 2k+1$.
\end{theorem}

\begin{proof}
It is clear that $1 \Rightarrow 2$. To prove the converse implication consider a trivialization $\sigma_0:L \to \mathbb{V}_{n,n-k}$ of the normal bundle $\nu(L)$ associated to a CI equation $\Xi:\EE^n\to \EE^{n-k}$ of $L$. Assuming that $L$ is also WI, its normal characteristic class vanishes by Proposition \ref{prop.normal}, and then Lemma \ref{inject} implies that $[\sigma]=0$ for any trivialization $\sigma$ of $\nu(L)$ because $n\leq 2k$. In particular, $\sigma_0$ extends over $\EE^n$, which means that $L$ is a tame complete intersection. Implication $2 \Rightarrow 1$ then follows by Theorem \ref{thm.main2}.

Finally, since any parallelizable $k$-submanifold of $\EE^n$ is WI for $n\geq 2k+2$ and also a complete intersection \cite{BK96}, we conclude from Theorem \ref{thm.SIgreat} and the previous discussion that Question~\ref{conjci} has an affirmative answer for $n\neq 2k+1$.
\end{proof}

\begin{remark}
There exist embeddings of exotic spheres of dimension $k=2^p$, $p\geq 4$, in $\EE^{2k-2}$ with trivial normal bundle which are not complete intersections~\cite{BK96}. These submanifolds are not parallelizable so they do not provide examples of WI submanifolds which are not SI.
\end{remark}

\begin{corollary}
Let $L$ be a $k$-submanifold of $\EE^n$ with $k+2\leq n\leq 2k$. If $(n-k) \in\{2,4,8\}$, then $L$ is SI if and only if it is WI.
\end{corollary}
\begin{proof}
The submanifold $L$ being WI or SI it has trivial normal bundle. As the codimension of $L$ is $2$, $4$ or $8$, it follows from a theorem of Bochnak and Kucharz \cite{BK96} that $L$ is a complete intersection. Applying Theorem \ref{T:lowdim} the claim is proved.
\end{proof}

\begin{corollary}
A $4$-dimensional parallelizable manifold $L$ admits a SI embedding in $\EE^n$ for any $n\geq 6$.
\end{corollary}
\begin{proof}
A parallelizable $4$-manifold always admits an embedding in $\EE^6$~\cite{CS79} whose normal bundle is trivial on account of Lemma \ref{L:trivnu}, and therefore this embedding is a complete intersection~\cite{BK96}. Then Theorem~\ref{T:classif} and Theorem~\ref{T:lowdim} imply that this embedding is SI and so will be any $n$-fattening for $n>6$.
\end{proof}

\begin{example} Any embedding $h:L \to \EE^n$ of a surface $L$ whose compact components are tori, is SI in the following cases:

\begin{enumerate}
\item $n=3$ and $L$ is open, by Corollary \ref{thm.cod1}. 

\item $n=4$ because any codimension two submanifold of $\EE^n$ is a complete intersection~\cite{BK96} and hence Theorem~\ref{T:lowdim} applies.

\item $n=5$ and $L$ is connected or open by Theorem \ref{thm.SIgreat}.

\item $n\geq 6$ again by Theorem \ref{thm.SIgreat}.
\end{enumerate}
\end{example}

The problem of characterizing SI embeddings in $\EE^{2k+1}$ remains open in general, except for the cases covered by Theorem \ref{thm.SIgreat}. In addition, it is not known whether any parallelizable submanifold of $\EE^{2k+1}$ with several compact connected components is a complete intersection. In the next section we shall focus on embeddings in $\EE^{2k+1}$, and in particular we shall provide examples of embeddings which are WI and CI but not SI.

\section{Strong integrability of $k$-manifolds in $\EE^{2k+1}$}\label{S:miyoshi}

In this section $L$ will be an oriented parallelizable manifold of dimension $k$. If $k\notin \{3,7\}$, then Theorem~\ref{T:classif} implies that $L$ is TWI, and hence any embedding of a connected manifold $L$ in $\EE^{2k+1}$ is SI provided that $k>1$ on account of Theorem~\ref{thm.SIgreat}. Therefore we shall assume in what follows that $k\in\{1,3,7\}$. In Subsections~\ref{subsec1} and~\ref{subsec2} we shall deal with the case $k=1$, while the remaining cases are studied in Subsection~\ref{S:SIk}.
 
In~\cite[Theorem 1]{Mi95} S. Miyoshi claimed the following result, which we restate using our terminology.

\begin{claim}[Miyoshi's ``theorem'']\label{thm.miyoshi}
Let $M$ be an open oriented $3$-manifold. A link $L\subset M$ is SI if and only there exists an orientation $\epsilon$ of $L$ such that $0=[L^\epsilon]\in H^\infty_1(M;\ZZ)$. In particular, any link is SI in $M=\EE^3$.  
\end{claim}

The main tools of Miyoshi's ``proof'' of Claim~\ref{thm.miyoshi} were the relative version of Phillips-Gromov $h$-principle, cf. Theorem~\ref{Gromov}, and his extension lemma~\cite[Lemma 3.1]{Mi95} which shows the existence of a CI equation for any link $L$ satisfying $0=[L^\epsilon]\in H^\infty_1(M;\ZZ)$. A careful analysis of Miyoshi's demonstration reveals that there is a substantial gap because he did not check whether the trivialization associated to a CI equation of $L$ is tame in the sense described by our Definition~\ref{D:ci}. Accordingly, his ``proof'' of Claim~\ref{thm.miyoshi} is not conclusive. As an aside remark, let us also mention that his proof of his second main theorem~\cite[Theorem 2]{Mi95} is not complete either because he did not prove that the normal bundle of $L$ extends as a trivial subbundle of $T(M)$; nevertheless the result holds true as proved in Application~\ref{appl.low}. 

The goal of Subsections~\ref{subsec1} and~\ref{subsec2} is to see that Claim~\ref{thm.miyoshi} does not hold in general, and also to provide a complete characterization of links in $\EE^3$ which are SI. In Subsection~\ref{S:SIk} we shall characterize SI connected manifolds in $\EE^{2k+1}$ for $k\in\{3,7\}$ using the semicharacteristic $\chi^*(L)$. In addition, we will show some examples of manifolds which are WI and CI but not SI ($\e^k$ in $\EE^{2k+1}$ for $k\in\{1,3,7\}$) and of manifolds whose SI character depends on the embedding (links in $\EE^3$). Let us observe that we have not handled the case of links in general open $3$-manifolds.

\subsection{Preferred trivializations and Seifert surfaces for links}\label{subsec1}

In this subsection $L$ will be an oriented (smooth) link in $\EE^3$, i.e. a disjoint union of embedded oriented circles. We denote by $\{L_i\}_{i=1}^r$ the connected components of $L$, and for each $i$, $N_i$ stands for an oriented (closed) tubular neighborhood of $L_i$ (with $N_i\cap N_j=\emptyset$ if $i\neq j$) and $T_i:=\partial N_i$ inherits an orientation because it is a boundary. Let us denote by $h_i:\e^1\to \EE^3$ an embedding such that $h_i(\e^1)=L_i$ and by $h$ the sum of all $h_i$, its image being $L$. Setting $T:=\bigcup_{i=1}^r T_i$, $N_L:=\bigcup_{i=1}^r N_i$ and $V:=\EE^3\backslash \text{int}\,N_L$, let us denote by $p_i:T_i\to N_i$ and $q:T\to V$ the obvious inclusion embeddings.

For each $1\leq i\leq r$ we choose a dual pair $(\alpha_i,\beta_i)$ of embedded oriented circles in $T_i$ as follows: $\beta_i$ is the oriented boundary of a fiber of the fibration $N_i\to L_i$ and $\alpha_i$ is such that it intersects with $\beta_i$ in exactly one point, this intersection being transversal and the pair $(\alpha_i,\beta_i)$ defining the orientation of $T_i$. Such a pair $(\alpha_i,\beta_i)$ defines a basis of $H_1(T_i,\ZZ)$ and $\beta_i$ being fixed as above, any two choices for the first generator are related by a homotopical relation: 
\begin{equation}\label{eq.alpha}
[\alpha'_i]=[\alpha_i]+m_i[\beta_i]\,,\,\,\,\,\,\,m_i\in \ZZ\,.
\end{equation}  
Note that any embedded circle $\alpha_i\subset T_i$ which is isotopic to $L_i$ in $N_i$, produces a trivialization of $T_i$ and is homotopic to a trivialization of $N_i$. We say that this trivialization of $N_i$ is \emph{associated} to $\alpha_i$.

Now we introduce the definition of a preferred trivialization (also used by Miyoshi). We use the notation introduced in Construction~\ref{const.trivial} without further mention.

\begin{definition}\label{def.preferred}
We say that a trivialization $\theta$ of a closed tubular neighborhood $N_L$ of $L$ is {\em preferred} if any fiber $\alpha$ of the associated trivialization $\phi:T\to \e^1$ is homologous to zero in $V$, that is $0=q_*([\alpha])\in H_1(V,\ZZ)$.
\end{definition}    

The following lemma establishes the existence and uniqueness of preferred trivializations for links.  

\begin{lemma}\label{L:preferred}
For any oriented link $L\subset \EE^3$ there exists, up to isotopy, exactly one preferred trivialization of $N_L$. 
\end{lemma}
\begin{proof}
The Mayer-Vietoris sequence for homology with integer coefficients yields
$$
\xymatrix{
0=H_2(\EE^3) \ar[r] & H_1(T)= \ZZ^{2r} \ar[r]^{p} & [\oplus_{i} H_1(N_i)] \oplus H_1(V) \ar[r]& H_1(\EE^3) = 0\,,}
$$
thus implying that $p=(\oplus_i p_{i*}) \oplus q_*$ is an isomorphism and $H_1(V;\ZZ)=\ZZ^r$. The lemma then stems from the following three observations:
\begin{enumerate}
\item For each $1\leq i\leq r$ we have $0 = p_{i*}([\beta_i]) \in H_1(N_i;\ZZ)$ and since $(\oplus_i p_{i*}) \oplus q_*$ is an isomorphism, we conclude that $0 \neq q_*([\beta_i]) \in H_1(V;\ZZ)$.

\item Since the map $q_*: H_1(T;\ZZ)\to H_1(V;\ZZ)$ is onto, we conclude that for each $i\in\{1\,\cdots,r\}$ we can choose an embedded circle $\alpha_i\subset T_i$ in such a way that $q_*([\alpha_i])=0$. This shows that the trivialization of $N_L$ associated to $\alpha=\bigcup_i \alpha_i$ is preferred by Definition~\ref{def.preferred}.
    
\item To prove the uniqueness, let us assume that there are $\alpha_i\subset T_i$ and $\alpha'_i\subset T_i$ for all $1\leq i\leq r$ with 
$$
q_*([\alpha_i])= q_*([\alpha'_i])=0\,,\,\,\,\,\,\,\,\,i\in\{1,\cdots,r\}\,.
$$
On account of Eq.~\eqref{eq.alpha} we deduce that $q_*(m_i [\beta_i]) = 0$ and hence $m_i=0$ for each $i$. It follows that the two sets of generators $\{\alpha_i\}_{i=1}^r$ and $\{\alpha'_i\}_{i=1}^r$ define pairwise the same homotopy classes, and so are isotopic.
\end{enumerate}  
\end{proof}

An appropriate way to describe this unique (up to isotopy) preferred trivialization of $N_L$ will be using the Seifert surfaces of $L$, so let us finish this subsection by recalling their construction~\cite[Chapter 5A]{Ro75}. 

\begin{const}(\textit{Seifert surfaces and associated preferred trivializations}).\label{const.seifert}

(1) For any oriented link $L\subset\EE^3$, there exists a compact oriented embedded surface $S\subset \EE^3$, which can be assumed to be connected, such that $L=\partial S$ as singular chains. Assuming that $S$ is transverse to $\partial N_L$, we observe that for each $i$, $\alpha_i:=S\cap T_i$ is an embedded circle isotopic to $L_i$ which defines a trivialization $\theta_i$ of $N_i$. This gives rise to a {\em Seifert trivialization} $\theta$ of $N_L$ which is preferred because its typical fiber $\alpha=S\cap T$ is homologous to zero in $V$, i.e. $0=q_*([\alpha])\in H_1(V,\ZZ)$. Since preferred trivializations of links are unique (Lemma~\ref{L:preferred}), the Seifert trivialization does not depend on the specific Seifert surface at hand.

(2) Let $(x,y,z)$ be the canonical coordinate system in $\EE^3$. By a general position argument, we can assume that the natural projection $\pi:\EE^3\to\EE^2$ given by $\pi(x,y,z)=(x,y)$ is such that $h'=\pi \circ h$ is an immersion in general position with respect to itself, that is its image $L'=\pi(L)$ has finitely many transverse crossings and each crossing point is the intersection of only two branches. Let us set $h'_i=\pi\circ h_i$ and $L'_i=\pi(L_i)$, then $h'$ and $h'_i$ are compressions of $h$ and $h_i$, cf. Construction~\ref{const.compression}, whose $3$-fattenings are denoted by $g$ and $g_i$ respectively. It is easy to check that $h$ and $g$, and $h_i$ and $g_i$, are regularly homotopic.

(3) Near each crossing point of $L'$, delete the over- and under-crossings and replace them by short-cut arcs, minding the orientation. This way we have a disjoint collection of $\tilde r$ simple closed oriented curves $\{\gamma_j\}_{j=1}^{\tilde r}\subset \{z=0\}$. Each circle bounds a disk $S_j$ in the plane and, although they may be nested, these disks can be made disjoint by pushing them slightly off the plane, starting with innermost ones and working outwards. The boundary of each disk $S_j$ is oriented by $L$, thus each $S_j$ inherits a natural orientation denoted by $\epsilon_j = \pm 1$ according to the fact that this orientation coincides or not with the canonical orientation of $\{z=0\}$. Finally, we connect these disks together at the old crossings with half-twisted strips to form a surface $S$ whose boundary is isotopic to the original link $L$, and after using a connected sum procedure, it is possible to assume that $S$ is connected.

(4) The procedure described in item (3) clearly applies to the case of a knot. In particular, we have a Seifert surface for each component $L_i$ of $L$, the simple closed oriented curves introduced above being denoted by $\gamma^i_j\subset \{z=0\}$, $1\leq j\leq r_i$, and the disks that they bound in the plane inheriting orientations $\epsilon^i_j=\pm 1$. It is important to notice that there is no obvious relationship between the Seifert surface of the link and those of its components.

(5) Finally we associate to the Seifert surface $S$ a normal frame $\mathcal S$ of $L$, called the {\it Seifert frame}. Calling $Z$ the unique vector field tangent to $L$ and coherent with its orientation, we take a vector field $X$ which is normal to $L$ and tangent to the Seifert surface and set $\mathcal S=(X,Z\wedge X)$. This frame defines an infinitesimal trivialization of the normal bundle  $\nu(L)$ of $L$ which gives a {\it Seifert normal map} $\sigma_h: L\to \VV_{32}$. As before, this map is unique because preferred trivializations are unique up to isotopy.
\end{const}

\subsection{The characterization of SI knots and links}\label{subsec2}

In this subsection, as in the previous one, $L$ will be an oriented (smooth) link in $\EE^3$. The first important observation is that any trivialization associated to a CI equation of $L$ is preferred. Conversely, proceeding as in the proof of Proposition~\ref{l.extension}, it can be readily derived that a preferred trivialization extends as a CI equation for any $L$ in $\EE^3$. Taking this fact into account, the following result is a straightforward consequence of Theorem~\ref{thm.main2} and Proposition~\ref{prop.normal}.   

\begin{proposition}\label{prop:preferred1}
$h$ is SI if and only if $N_{L}$ admits a preferred (or Seifert) trivialization whose associated normal map $\sigma_h$ is homotopic to a constant, i.e. $[\sigma_h]=0$. This happens if and only if $[\sigma_{h_i}]=0$ for all i. 
\end{proposition}

To characterize the links which are SI we need to use two trivializations of $\nu(L)$: the Seifert trivialization and the {\it curvature trivialization} that we introduce as follows.  

\begin{const}
Let $\hat Y$ be a normal vector field to $L'\subset \EE^2$, and take the horizontal lift vector field $Y$ on $L$ such that $\pi_*(Y_p)=\hat Y_{\pi(p)}$ for any $p\in L$. The {\it curvature frame} $\mathcal C=(Y,Z\wedge Y)$ defines an infinitesimal trivialization of the normal bundle  $\nu(L)$ of $L$ which gives a {\it curvature normal map} $\tilde\sigma_h: L\to \VV_{32}$. It is obvious that this map is unique up to homotopy.
\end{const}

\begin{lemma}\label{lem.curvature}
For each $i\in\{1,\cdots,r\}$, $[\tilde \sigma_{h_i}]=\sum_{j=1}^{r_i}\epsilon^i_j$ mod. $2$. 
\end{lemma}
\begin{proof}
Let us first compute the curvatura integra $Ci(h'_i)$ of $h'_i$. By definition, $Ci(h'_i)$ is the degree of the normal map $\sigma_{h'_i}:L'_i\to \e^1$. We note that the immersion $h'_i$ is regularly homotopic to an immersion $h''_i$ obtained by deforming $h'_i$ so that the two branches of $L'_i$ near a crossing point become infinitesimally tangent. Of course $Ci(h''_i) = Ci(h'_i)$. 

The image $L''_i$ of $h''_i$ decomposes into $r_i$ smoothly embedded circles $\tilde\gamma^i_j$ homotopic to the corresponding circles $\gamma^i_j$. It is clear that $Ci(h''_i) = \sum_{j=1}^{r_i} Ci(\tilde\gamma^i_j)$ and $Ci(\tilde\gamma^i_j) = \epsilon^i_j$ for each $j$ because the orientations of the disks bound by $\tilde\gamma^i_j$ and $\gamma^i_j$ are the same. Now it is easy to prove that $[\sigma_{g_i}] = Ci(h'_i)$ mod. $2$ by using the same argument as in the proof of Proposition~\ref{spheres}. The lemma then follows by recalling that $g_i$ and $h_i$ are regularly homotopic and noticing that $\tilde \sigma_{h_i}$ is homotopic to $\sigma_{g_i}$ by construction.
\end{proof}
 
In order to compute the class of the Seifert normal map $[\sigma_{h}]$, we fix the trivialization of the normal bundle $\nu(L)$ given by the curvature frame $\mathcal C$. At each point of $L$ the Seifert frame $\mathcal S$ and the curvature frame $\mathcal C$ are related by a normal twist, which can be precisely described by the following description. The standard fibration $SO(2)\to \VV_{3,2}\to \GG_{3,2}$ yields the homotopy exact sequence:
$$
\xymatrix{
\ldots\ar[r] & \pi_1(SO(2))=\ZZ \ar[r]^s & \pi_1(\VV_{3,2})=\ZZ_2 \ar[r] & \pi_1(\GG_{3,2})=0 \ar[r] & \ldots
}
$$
The map $s$ is onto and therefore for each $i\in\{1,\cdots,r\}$ there exists a map $\psi_i:L_i\to SO(2)$ such that 
\begin{equation}\label{eq.cs}
[\sigma_{h_i}]=[\tilde \sigma_{h_i}]+[\psi_i]\,\,\, \text{mod}.\,\,\, 2\,.
\end{equation}
We call $[\psi_i]\in \pi_1(\e^1)=\ZZ$ the {\it relative homotopy class} of $[\sigma_{h_i}]$ with respect to $[\tilde \sigma_{h_i}]$, and we denote it by $[\mathcal S_i,\mathcal C_i]$. 

The aim of the next lemma is to compute these relative homotopy classes $[\mathcal S_i,\mathcal C_i]$, but let us first introduce some notation. 

\begin{notation}\label{not.sign}
(1) We assign a sign to each crossing of $L'$ in the usual way: at each crossing point we have the direction $v_1$ of the over-branch and the direction $v_2$ of the under-branch; if this couple of vectors $\{v_1,v_2\}$ defines the same orientation as the canonical basis of $\RR^2$ we assign the sign $+1$ to the crossing, and $-1$ otherwise. 

(2) For each $i\in\{1,\cdots,r\}$ we denote by $\{q^i_{jk}\}_k$ the set of crossing points of $L'_i$ with $L'_j$, so that when $i=j$ it gives the set of crossings of $L'_i$ with itself. Similarly, we call $n^j_i$ the number of $+1$ crossings of $L'_i$ with $L'_j$, and $m^j_i$ the number of $-1$ crossings of $L'_i$ with $L'_j$. It is well known~\cite[Chapter 5D]{Ro75} that the linking number of $L_i$ with $L_j$ for $j\neq i$ is given by:
\begin{align}\label{eq.link}
\text{lk}(L_i,L_j)=\frac{n^j_i-m^j_i}{2}\,.
\end{align}
For $i=j$ the number $W_i:= n^i_i-m^i_i$ is defined as the Whitney number~\cite[Section 14.3]{DFN85} of $L_i$.
\end{notation}

\begin{lemma}\label{lem.relative}
For each $i\in\{1,\cdots,r\}$, $[\mathcal S_i,\mathcal C_i]=W_i+\sum_{j\neq i}\text{lk}(L_i,L_j)$.
\end{lemma}
\begin{proof}
We consider disjoint neighborhoods $\hat A^i_{jk}$ of the crossing points $q^i_{jk}$, and set $A^i_{jk}:=\pi^{-1}(\hat A^i_{jk})$. We can assume without loss of generality that the Seifert vector field $X$ is horizontal and collinear (but possibly of opposite direction) with the curvature vector field $Y$ at each point of $L\backslash\bigcup_{j,k} A^i_{jk}$. Therefore the windings of $X$ with respect to $Y$ are located in each $A^i_{jk}$. By Construction~\ref{const.seifert} of the Seifert surface for $L$, $X$ twists by $\pm 1/2$ along each branch of $L_i$ when passing through $A^i_{jk}$, the sign being given by the orientation of the corresponding crossing point $q^i_{jk}$, see Notation~\ref{not.sign}. Therefore, the vector field $X$ makes a total twist $\pm 1$ with respect to $Y$ in $A^i_{ik}$ and a half twist of $\pm 1/2$ in $A^i_{jk}$ for $j\neq i$. According to the notation introduced before we conclude that:
$$[\mathcal S_i,\mathcal C_i]=n^i_i-m^i_i+\sum_{j\neq i}\frac{n^j_i-m^j_i}{2}\,,$$
and the lemma follows by Eq.~\eqref{eq.link} and the definition of the Whitney number.  
\end{proof}

\begin{corollary}\label{lem.link}
For each $i\in\{1,\cdots,r\}$, the Seifert normal map $\sigma_{h_i}:L_i\to \VV_{32}$ verifies that $[\sigma_{h_i}]=1+\sum_{j\neq i}\text{lk}(L_i,L_j)$ mod. $2$.
\end{corollary}
\begin{proof}
Lemmas~\ref{lem.curvature} and~\ref{lem.relative} and Eq.~\eqref{eq.cs} imply that $$[\sigma_{h_i}]=\sum_{j=1}^{r_i}\epsilon^i_j+W_i+\sum_{j\neq i}\text{lk}(L_i,L_j)\,\,\,\,\,\,\text{mod}\, 2\,.$$ The claim then follows by noticing that $\sum_{j=1}^{r_i}\epsilon^i_j+W_i=1$ mod. 2, see e.g.~\cite[Section 14.3]{DFN85}.

\end{proof}

Let us observe that in the proof of Lemma~\ref{lem.relative} there are no more vertical twists than the ones associated to the self-crossings and the crossings with other components. This is the key distinction with a trivialization which is not preferred. Extra twists may make a trivialization have zero normal class.

\begin{theorem}\label{thm.silink}
A link in $\EE^3$ is SI if and only if $$\sum_{j\neq i}\text{lk}(L_i,L_j)=1\,\,\,\text{mod.}\,\,\,2\,$$
for all $i\in\{1,\cdots,r\}$.
\end{theorem}
\begin{proof}
It is a straightforward consequence of Proposition~\ref{prop:preferred1} and Corollary \ref{lem.link}. 
\end{proof}

\begin{remark}
Theorem~\ref{thm.silink} implies that the zero set of any WI equation of a knot has at least other component, open or compact, which is linked with the knot. It is easy to prove that any link  admits a WI equation whose zero set contains exactly $r$ compact components, which are of course $L$.
\end{remark}

\begin{corollary}\label{cor.nosi}
No knot in $\EE^3$ is SI. This contradicts Miyoshi's Claim~\ref{thm.miyoshi}.
\end{corollary}

We conclude that links in $\EE^3$ are examples of manifolds whose SI character depends on the embedding, in strong contrast with the WI character, which does not depend on the embedding. For example, any unlinked union of knots, the Borromean rings and the Whitehead link are not SI, but any link consisting of two knots which have odd linking number is SI.

\subsection{Connected $k$-manifolds in $\EE^{2k+1}$ for $k\in\{3,7\}$}\label{S:SIk}

In this subsection we show that some of the ideas used to study the strong integrability of links in $\EE^3$ can be extended to analyze embeddings of parallelizable connected $k$-manifolds $L$ in $\EE^{2k+1}$ for $k\in\{3,7\}$. Recall from Lemma~\ref{L:trivnu} that any of these embeddings have trivial normal bundle. First, we need to extend the notion of preferred trivialization of a tubular neighborhood (compare with Definition~\ref{def.preferred}). For the sake of notational simplicity we shall usually identify $L$ with its embedded image $h(L)$ in $\EE^{2k+1}$.

\begin{definition}
Let $L$ be a $k$-submanifold of $\EE^{2k+1}$ with trivial normal bundle. We say that a trivialization $\theta$ of a closed tubular neighborhood $N_L$ of $L$ is {\em preferred} if any fiber $L'$ of the associated trivialization $\phi:\partial N_L\to \e^k$ is homologous to zero in $V:=\EE^{2k+1}\backslash \text{int}\,N_L$, that is $0=[L']\in H^\infty_k(V;\ZZ)$ when endowed with the corresponding orientation.
\end{definition} 

In the next lemma we establish the uniqueness (up to isotopy) of preferred trivializations of a tubular neighborhood of $L$. The proof is analogous to the proof of Lemma~\ref{L:preferred}, so details will be omitted.

\begin{lemma}\label{lem.unique}
For any embedding $h:L\subset \EE^{2k+1}$ there exists, up to isotopy, exactly one preferred trivialization of $N_L$. 
\end{lemma} 
\begin{proof}
Any such a manifold is CI by~\cite{BK96}, and hence any CI equation endows a tubular neighborhood $N_L$ of $L$ with a preferred trivialization. The uniqueness of this preferred trivialization follows exactly, mutatis mutandis, proceeding as in the proof of Lemma~\ref{L:preferred}.
\end{proof}  

The following result is a straightforward consequence of Theorem~\ref{thm.main2} and Proposition~\ref{prop.normal}. It shows that the direct implication of Proposition~\ref{prop:preferred1} holds.

\begin{proposition}\label{prop:preferred}
Let $L$ be a $k$-manifold. If an embedding $h:L\to \EE^n$ is SI then $N_{L}$ admits a preferred trivialization whose associated infinitesimal trivialization $\sigma_h$ is homotopic to a constant, i.e. $[\sigma_h]=0$.
\end{proposition}

The main theorem of this subsection is a complete characterization of SI embeddings of connected manifolds of dimension $k\in\{3,7\}$ in $\EE^{2k+1}$. In particular, it provides examples of embeddings which are WI and CI but not SI (compare with Question~\ref{conjci}). Its corollary below extends Corollary~\ref{cor.nosi} to embeddings of higher dimensional spheres and proves Theorem~E of Section~\ref{S.mostrelevant}.

\begin{theorem}\label{thm.sk}
Let $L$ be a connected $k$-manifold which is parallelizable, $k\in\{3,7\}$. Then the following conditions are equivalent:
\begin{enumerate}
\item There exists an embedding $h:L\to \EE^{2k+1}$ which is SI.
\item Any embedding $h:L\to \EE^{2k+1}$ is SI.
\item The semicharacteristic of $L$ is an even number, i.e. $\chi^*(L)=0$ mod. $2$.
\end{enumerate}
Moreover any embedding $h:L\to \EE^{2k+1}$ is WI and CI. 
\end{theorem}
\begin{proof}
Theorem~\ref{t.aob} and Ref.~\cite{BK96} imply that any embedding $h:L\to \EE^{2k+1}$ is WI and CI. On the other hand, Wu's Theorem~\ref{T:Wu} ensures that all the embeddings $h:L\to \EE^{2k+1}$ are isotopic, so this proves $1\Leftrightarrow 2$. The fact that $3\Rightarrow 2$ follows from Theorems~\ref{thm.caract} and~\ref{thm.SIgreat}. To prove that $2\Rightarrow 3$ we proceed in three steps:

(1) Since all the embeddings $h:L\to \EE^{2k+1}$ are isotopic, we can assume that $h$ is the $(2k+1)$-fattening of an embedding $h':L\to \EE^{2k}$ with trivial normal bundle, which exists by Whitney's theorem~\cite{Wh44}. Let $\sigma_{h'}:L\to \VV_{2k,k}$ be a trivialization of the normal bundle of $h'$. Following Construction~\ref{const.fattening}, $\sigma_{h'}$ defines a trivialization $r\circ \sigma_{h'}=\sigma_h:L\to \VV_{2k+1,k+1}$ of the normal bundle of $h$ using the fattening sequence
$$
\xymatrix{
\VV_{2k,k}\ar[r]^r & \VV_{2k+1,k+1} \ar[r] & \e^{2k}
}\,.
$$
Since its associated homotopy exact sequence is
$$
\xymatrix{ \pi_{k+1}(\e^{2k})=0 \ar[r] & \pi_{k}(\VV_{2k,k}) \ar[r]^{r_*}& \pi_{k}(\VV_{2k+1,k+1}) \ar[r]& \pi_{k}(\e^{2k})=0 \,,}
$$
with $\pi_{k}(\VV_{2k,k})=\pi_{k}(\VV_{2k+1,k+1})=\ZZ_2$, we conclude that $r_*$ is an isomorphism and $[\sigma_h]=[\sigma_{h'}]$.

(2) Let $f:L\to \EE^{k+1}$ be a compression of $h'$. Proceeding as in the proof of Proposition~\ref{prop.key}, it follows from Lemma~\ref{inject} and the fact that $[\sigma_{h'}]\in \ZZ_2$ that $[\sigma_{h'}]=Ci(f)$ mod. $2$ for any trivialization of $\nu(h')$. Therefore Proposition~\ref{prop.3and7} and item $1$ imply that $[\sigma_h]=[\sigma_{h'}]=\chi^*(L)$ mod. $2$, and this does not depend on the trivialization of $\nu(h')$.

(3) We claim that the trivialization $\sigma_h$ obtained by fattening $\sigma_{h'}$ is preferred for any trivialization of $\nu(h')$. Indeed, since $h:L\to \EE^{2k}\times \EE$ is transverse to the factor $\EE$ and $\sigma_h$ is obtained from $\sigma_{h'}$ by adding a vertical vector, we have that any fibre of the associated differentiable trivialization $\phi_h$ is isotopic in $\partial N_L$ to $F:=\partial N_L\cap (L\times\EE)$, and this does not depend on $\sigma_{h'}$. The trivialization $\phi_h$ is hence preferred because $0=[F^\epsilon]\in H_k^\infty(V;\ZZ)$ when $F$ is endowed with the orientation $\epsilon$ induced by $L$.  The implication $2\Rightarrow 3$ then follows from the uniqueness of preferred trivializations, cf. Lemma~\ref{lem.unique}, and Proposition~\ref{prop:preferred}.    
\end{proof}

\begin{corollary}
No embedding of $\e^k$, $k\in\{1,3,7\}$, in $\EE^{2k+1}$ is SI.
\end{corollary}  
\begin{proof}
The case $k=1$ is just Corollary~\ref{cor.nosi}. The other cases follow from Theorem~\ref{thm.sk}.
\end{proof} 

\section{Applications to foliation theory}\label{Euc6}

As explained in Section \ref{S:pre}, a WI submanifold $L$ is a union of leaves of the simple foliation $\cF_{\Phi}$ defined by an arbitrary weak equation $\Phi$ of $L$. In Euclidean spaces a converse statement holds, thus establishing the equivalence of these two properties.

\begin{proposition}\label{P:foliat}
A submanifold $L\subset \EE^n$ of codimension $m\geq 2$ is WI if and only if it is a union of proper leaves of some codimension $m$ foliation of $\EE^n$.
\end{proposition}
\begin{proof}
If $L$ is a union of proper leaves in $\EE^n$, we can extend the normal bundle $\nu(L)$ of $L$ over $\EE^n$ by the normal bundle of the foliation, and hence $L$ is WI by Proposition \ref{prop.normal}. The converse statement is trivial.
\end{proof}

In light of Proposition \ref{P:foliat}, the problem of realizability of $L$ as a union of proper leaves in $\EE^n$ is equivalent to the study of its weak integrability, which seems a priori to be a much stronger requirement. In particular, all the results of our Chapter~\ref{Ch:Euc} can be translated into this context of realizability, however it is important to notice that this approach settles the question of realizability only for compact manifolds: it does not take into account the fact that an open manifold $L$ could possibly be a non-proper leaf. The following is our most significant result in this context, it includes Theorem F and Corollary G of Section~\ref{S.mostrelevant}.

\begin{theorem}\label{T:realization} Let $L$ be a compact and connected $k$-dimensional submanifold of $\EE^n$. Then:
\begin{enumerate}
\item If $k\notin\{3,7\}$, $L$ is a leaf of a foliation of $\EE^n$, $n\geq k+2$, if and only if it is parallelizable and has trivial normal bundle.
\item If $k \in \{3,7\}$, $L$ is a leaf of a foliation of $\EE^n$, $n\leq 2k$, if and only if it is parallelizable, it has trivial normal bundle and $\chi^*(L)=0$ mod. $2$. For example, no $\ZZ_2$-homology sphere of dimension $k$ can be realized as a leaf of a foliation of $\EE^n$ for $n\leq 2k$.
\item For any $k\geq 1$, $L$ is a leaf of a foliation of $\EE^n$, $n\geq 2k+1$, if and only if it is parallelizable.
\end{enumerate}
\end{theorem}

\begin{proof}
Taking into account Proposition \ref{P:foliat}, statement $1$ follows from Theorem~\ref{T:classif}, statement $2$ is proved using Theorem~\ref{thm.caract}, and statement $3$ is a consequence of Theorem~\ref{t.aob}.
\end{proof}

Note that for $k \in \{3,7\}$, the sphere $\e^k$ is easily realized as a leaf of a simple foliation of $\EE^n$, $n\geq 2k+1$, using the quaternionic or octonionic Hopf fibration $\e^k\to \e^{2k+1}\to \e^{k+1}$ and identifying $\EE^{2k+1}$ with the complement of a point $\{p\}$ in $\e^{2k+1}$. Note also that Theorem \ref{T:realization} provides a new insight into Vogt's question \cite{V93} asking which closed 3-manifolds occur as leaves of a foliation of $\EE^5$ with all leaves compact (but a priori not necessarily diffeomorphic).

\begin{appl}
Let $p$ be an odd number and $L_p:=L_{p,1}\sharp L_{p,1}$ be the connected sum of two copies of the lens space $L_{p,1}$. It can be proved \cite{GiLi83} that $L_p$ embeds into $\EE^4$ and thus so does (properly) the non-compact manifold $L_*:=\sharp_{p\in\,\text{odd}}\,L_p$. According to Application~\ref{appl.caract} and Corollary~\ref{C:sums}, $L_p$ is not TWI so it cannot be a leaf in $\EE^n$ for $n\leq 6$, but Corollary \ref{thm.cod1} implies that $L_*$ is SI and hence a leaf of a simple codimension one foliation of $\EE^4$. In contrast, $L_*$ cannot be a leaf, proper or not, of a codimension one foliation of any closed $4$-manifold \cite{Gh85,In85}. 

In general, any open codimension $m\geq 1$ submanifold properly embedded in $\EE^n$ with trivial normal bundle, is a leaf of a simple foliation of $\EE^n$, cf. Theorem~\ref{thm.open} and Proposition~\ref{P:foliat}.
\end{appl}

We would also like to make some remarks about applications of our results to a field which has aroused some interest in the last few years. First we need a definition.

\begin{definition}\label{D:crit}
A closed set $K\subset \EE^n$ is \textit{critical} if there exists a smooth function $F:\EE^n\to\EE$ such that $K=\{x\in\EE^n:\dd F(x)=0\}$.
\end{definition}

\begin{proposition}\label{P:critical}
A submanifold $L\subset \EE^n$ of codimension $m\geq 1$ which is SI is also critical.
\end{proposition}
\begin{proof}
 A strong equation $\Phi:\EE^n\to \EE^m$ of $L$ defines $m$ coordinate functions $\Phi_i:\EE^n\to\EE$ and $L$ is critical with respect to the function $F:=\sum_{i=1}^m \Phi_i^2$.
\end{proof}

Using Proposition \ref{P:critical} and the results of Section \ref{Euc5} the following corollary follows immediately. It extends a classical theorem in the subject which states that any tame link in $\EE^3$ is critical~\cite{GP93}.

\begin{corollary} Any torus $\TT^2$ in $\EE^n$, $n\geq 4$, is critical. Any parallelizable $k$-submanifold $L \subset \EE^n$, $n \geq 2k+2$, is also critical.

\end{corollary}

\chapter{Final remarks}\label{Ch:final}

\section{Analytic manifolds and integrable embeddings}

In the entire paper we have worked in the smooth category, i.e. weak and strong equations are maps $\Phi\in C^\infty(M,\EE^m)$. Indeed many of the techniques we have used, e.g. extension theorems and the relative Phillips-Gromov h-principle, are intrinsically $C^\infty$. Therefore it is natural to ask whether our results can be extended to the real-analytic class provided that $M$ is an analytic manifold. A deep theorem of Shiota allows us to settle this question quite easily.

\begin{proposition}
Let $L$ be an analytic submanifold of an analytic manifold $M$. If $L$ is WI [resp. SI], it admits a weak [resp. strong] equation of class~$C^\omega$.
\end{proposition}
\begin{proof}
Let $\Phi_0:M\to \EE^m$ be a smooth weak equation of $L$. The fact that $\Phi_0$ is a submersion implies \cite[Theorem 10.1]{Sh81} that there is a diffeomorphism $\phi_0\in \text{Diff}^\infty(M)$ such that $\Phi_1:=\Phi_0\circ \phi_0^{-1}$ is of class $C^\omega(M,\EE^m)$. Of course $\phi_0(L)\subset \Phi_1^{-1}(0)$ and since both $L$ and $\phi_0(L)$ are analytic submanifolds related by a smooth diffeomorphism $\phi_0$, there is an analytic diffeomorphism $\phi_1$ of $M$ such that $\phi_1(\phi_0(L))=L$ (see \cite[Corollary 8.6]{Sh81} or also \cite{HM62}). Then $\Phi:=\Phi_1\circ \phi_1^{-1}$ is a weak equation of $L$ of class $C^\omega$; it will be strong if $\Phi_0$ is strong.
\end{proof}

The extension of our results to the holomorphic setting is more complicated. In this context, $L$ is a closed complex submanifold of a Stein manifold $M$ and using a holomorphic version of the relative Phillips-Gromov h-principle \cite[Theorem 2.5]{Fo03}, it is possible to extend some of our results concerning weak integrability. The study of strong integrability is less straightforward because smooth complete intersections are not necessarily holomorphic complete intersections \cite{Fo01}, which gives rise to several difficulties. A careful analysis of the holomorphic situation is left to the interested reader.

Finally some words on the polynomial setting are in order. It is well known that any smooth compact submanifold $L\subset \EE^n$ embedded with trivial normal bundle is isotopic to a non-singular compact component of an algebraic set \cite[Section 14]{BCR98}. This general result is explicitly illustrated by Perron in \cite{Pe82} where he constructs a polynomial map $F:\EE^4\to\EE^2$ such that $F^{-1}(0)\cap B_\epsilon$ is a figure-eight knot in $B_\epsilon:=\{x\in\EE^4:|x|=\epsilon\}$, provided that $\epsilon>0$ is small enough. Using the stereographic projection and possibly multiplying by some non-vanishing global factor, it is then easy to obtain a polynomial map $G:\EE^3\to\EE^2$ for which $0$ is a regular value and $G^{-1}(0)$ is a figure-eight knot, but this map $G$ is not a submersion. More generally, if $L$ is WI we do not know whether it
admits or not a polynomial (weak) equation, possibly up to isotopy. In particular the following questions arise:

\begin{problem}
Are polynomial submersions $Sub_{\text{pol}}(\EE^n,\EE^m)$ dense in smooth submersions $Sub(\EE^n,\EE^m)$ in some $C^r$-weak topology?
\end{problem}
\begin{problem}
Does there exist, up to isotopy, a polynomial weak equation $\Phi:\EE^n\to \EE^{n-1}$ for any (finite) link $L\subset \EE^n$, $n\geq 3$?
\end{problem}

\section{Open problems}

To finish this work, we state some more questions which we have not been able to solve.

%

\begin{problem} A compact $3$-manifold can be realized as a leaf of a foliation in $\EE^5$ if and only if its semicharacteristic is $0$ mod. $2$. Characterize those which can be realized as a leaf of a compact foliation in $\EE^5$ (compare with Vogt's question~\cite{V93}).
\end{problem}

\begin{problem}

Let $L$ be a WI submanifold of $\EE^n$. The results in Sections~\ref{Euc5} and~\ref{S:miyoshi} imply that there are two kind of obstructions which prevent $L$ from being SI: 
\begin{enumerate}
\item If $n\leq 2k$, $L$ is not SI if and only if it is not CI. 
\item If $n=2k+1$, $L$ is not SI if and only if it is CI but not TCI or it has several connected components and is not CI.
\end{enumerate}

The knots and links in $\EE^3$ provide examples of the second kind. Therefore, for the first kind, it remains open to find WI submanifolds which are not complete intersections in $\EE^n$ for $n\leq 2k$. If would also be interesting to find connected manifolds whose CI (and hence SI) character in $\EE^n$, $n\leq 2k$, depends on the embedding (with trivial normal bundle). Finally, we notice that Question~\ref{conjci} remains open for $L\subset \EE^{2k+1}$ with several compact components.   
\end{problem}

\begin{problem}
The proofs of this paper are non-constructive and hence we do not provide any tool to obtain explicit weak or strong equations for a given submanifold $L$. In particular, it is open to construct an explicit $C^\omega$ weak equation $\Phi:\EE^3\to \EE^2$ (in terms of elementary functions) for a knot (or link) in $\EE^3$.
\end{problem}

\end{document}